\documentclass[11pt,a4paper]{article}

\usepackage{amssymb}

\usepackage{ntheorem} % statt amsthm
\usepackage[sumlimits,intlimits]{amsmath}        % links Nummern, Summen und Integrale i

\allowdisplaybreaks                % Seitenumbruch in Formeln ermï¿½lichen
\theorembodyfont{\rm}              % Aufrechtes setzen des Textes
\theorembodyfont{\slshape}
\newtheorem{lem}{Lemma}[section]
\newtheorem{thm}[lem]{Theorem}
\newtheorem{prop}[lem]{Proposition}
\newtheorem{cor}[lem]{Corollary}
\theorembodyfont{\rm}              % Aufrechtes setzen des Textes
\newtheorem{rem}[lem]{Remark}
\newtheorem{ex}[lem]{Example}

\newtheorem{defnr}[lem]{Definition \& Remark}
\newtheorem{defn}[lem]{Definition}
\newtheorem{nota}[lem]{Notation}

\newtheorem*{thmn}{Theorem}

\newenvironment{proof}{\noindent\textit{Proof. }}{\hfill $\Box$\\[0.5cm]}

\newcounter{fig}
\newcommand{\new}{\medskip\noindent(\thefig)  \addtocounter{fig}{1}}
\newcommand{\startenum}{\setcounter{fig}{1}}

\DeclareMathSymbol{\ordinaryl}{\mathalpha}{letters}{`l}     % erzeugt automatisch ein geschwungenes
\begingroup                             % l in der Mathematik-Umgebung
\lccode`\~=`\l
\lowercase{\gdef ~{\ifnum\the\mathgroup=-1 \ell \else \ordinaryl \fi}}
\endgroup
\mathcode `\l="8000

\newcommand{\bdefn}{\begin{defn}}
\newcommand{\edefn}{\end{defn}}
\newcommand{\bdefnr}{\begin{defnr}}
\newcommand{\edefnr}{\end{defnr}}
\newcommand{\benum}{\begin{enumerate}}
\newcommand{\eenum}{\end{enumerate}}
\newcommand{\bthm}{\begin{thm}}
\newcommand{\ethm}{\end{thm}}
\newcommand{\bnota}{\begin{nota}}
\newcommand{\enota}{\end{nota}}
\newcommand{\bproof}{\begin{proof}}
\newcommand{\eproof}{\end{proof}}
\newcommand{\bprop}{\begin{prop}}
\newcommand{\eprop}{\end{prop}}
\newcommand{\bcor}{\begin{cor}}
\newcommand{\ecor}{\end{cor}}
\newcommand{\blem}{\begin{lem}}
\newcommand{\elem}{\end{lem}}
\newcommand{\brem}{\begin{rem}}
\newcommand{\erem}{\end{rem}}
\newcommand{\bex}{\begin{ex}}
\newcommand{\eex}{\end{ex}}
\newcommand{\btab}{\begin{tabular}}
\newcommand{\etab}{\end{tabular}}
\newcommand{\bctab}{\begin{center}\begin{tabular}}
\newcommand{\ectab}{\end{tabular}\end{center}}

\newcommand{\ba}{\begin{array}}
\newcommand{\ea}{\end{array}}
\newcommand{\bea}{\begin{eqnarray}}
\newcommand{\eea}{\end{eqnarray}}
\newcommand{\bean}{\begin{eqnarray*}}
\newcommand{\eean}{\end{eqnarray*}}

\newcommand{\cs}{\check{S}}

\newcommand{\cbb}{\mathbb{C}}

\newcommand{\lbb}{\mathbb{L}}
\newcommand{\nbb}{\mathbb{N}}
\newcommand{\qbb}{\mathbb{Q}}
\newcommand{\rbb}{\mathbb{R}}

\newcommand{\acal}{\mathcal{A}}
\newcommand{\bcal}{\mathcal{B}}
\newcommand{\ccal}{\mathcal{C}}

\newcommand{\fcal}{\mathcal{F}}

\newcommand{\mcal}{\mathcal{M}}

\newcommand{\ocal}{\mathcal{O}}

\newcommand{\scal}{\mathcal{S}}
\newcommand{\ucal}{\mathcal{U}}

\newcommand{\ph}{pluri\-harmonic}
\newcommand{\sph}{sub\-pluri\-harmonic}

\newcommand{\psh}{pluri\-sub\-harmonic}

\newcommand{\qpsh}{$q$-pluri\-sub\-harmonic}

\newcommand{\sqpsh}{strictly $q$-pluri\-sub\-harmonic}

\newcommand{\qpshy}{$q$-pluri\-sub\-harmonicity}

\newcommand{\qhol}{$q$-holo\-morphic}

\newcommand{\psc}{pseudo\-convex}

\newcommand{\qpsc}{$q$-pseudo\-convex}

\newcommand{\PSH}{\mathcal{PSH}}

\newcommand{\USC}{\mathcal{USC}}

\newcommand{\nbh}{neighborhood}

\newcommand{\inti}{\mathrm{int}}

\newcommand{\cont}{continuous}
\newcommand{\usc}{upper semi-continuous}
\newcommand{\lsc}{lower semi-continuous}
\newcommand{\fct}{function}
\newcommand{\fcts}{functions}

\renewcommand{\and}{\ \mathrm{and}\ }
\newcommand{\qand}{\quad \mathrm{and} \quad}

\newcommand{\diam}{\mathrm{diam}}

\newcommand{\dd}{\partial}
\newcommand{\dbar}{\overline{\partial}}
\newcommand{\cld}{\overline{D}}
\newcommand{\ol}[1]{\overline{#1}}

\newcommand{\relc}{\Subset}

\newcommand{\rmax}[1]{\widetilde{\mathrm{max}}_{#1}}

\newcommand{\SA}{\check{S}_\acal}
\newcommand{\bA}{b_\acal}
\newcommand{\BA}{B_\acal}
\newcommand{\PA}{P_\acal}
\newcommand{\MA}{m_\acal}

\renewcommand{\Re}{\mathrm{Re}}

\newcommand{\nach}{\rightarrow}

\newcommand{\D}{\displaystyle}

\newcommand{\eps}{\varepsilon}
\newcommand{\vphi}{\varphi}
\newcommand{\vrho}{\varrho}

\newcommand{\cl}[1]{\ol{#1}}

\newcommand{\scl}[1]{\ol{#1}^\downarrow}
\newcommand{\kscl}[2]{\ol{#2}^{\downarrow #1}}

\author{T. Pawlaschyk \footnote{The author was partially supported by the German Research Foundation (DFG, Deutsche Forschungsgemeinschaft) under the project 'Pluripotential Theory, Hulls and Foliations', grant SH 456/1-1. The author would like to thank Prof. N.V.Shcherbina for his advisory during the writing process.}}

\title{The Bergman-Shilov boundary for subfamilies of $q$-plurisubharmonic functions}

\date{}

\begin{document}

\maketitle

\begin{abstract} We introduce a notion of the Bergman-Shilov (or Shilov) boundary for some subclasses of \usc\ \fcts\ on a compact Hausdorff space. It is by definition the smallest closed subset of the given space on which all \fcts\ of that subclass attain their maximum. For certain subclasses with simple structure one can show the existence and uniqueness of the Shilov boundary. Then we provide its relation to the set of peak points and establish Bishop-type theorems. As an application we obtain a generalization of Bychkov's theorem which gives a geometric characterization of the Shilov boundary for \qpsh\ \fcts\ on convex bounded domains. In the case of bounded \psc\ domains with smooth boundary we also show that some parts of the Shilov boundary for \qpsh\ \fcts\ are foliated by $q$-dimensional complex submanifolds.
\end{abstract}

%\addcontentsline{toc}{section}{Contents}
  \tableofcontents

\addcontentsline{toc}{section}{Introduction}

\section*{Introduction}

In his article \cite{By} from 1981, S.N.Bychkov gave a geometric characterization of the Shilov boundary for bounded convex domains in $\cbb^n$. The aim of our paper is to generalize his result to the Shilov boundary with respect to \qpsh\ and \qhol\ \fcts\ on bounded convex domains. These classes of \fcts\ were already studied by different authors, e.g., R.Basener in \cite{Ba}, R.L.Hunt and J.J.Murray in \cite{HM} or Z.S\l odkowski in \cite{Sl}, \cite{Sl2}. It was H.J.Bremermann in \cite{Br1} who observed that there is a characterization of a Bergman-Shilov boundary (or, for short, Shilov boundary) based on \psh\ \fcts\ without showing its existence. This gap was filled by, e.g., J.Siciak in \cite{Sic}. Given a compact Hausdorff space $K$ and a subclass $\acal$ of \usc\ \fcts\ on $K$, the Shilov boundary for $\acal$ is the smallest closed subset of $K$ on which all \fcts\ from $\acal$ attain their maximum. Existence and uniqueness for such a subset is guaranteed if $\acal$ has some simple structure, e.g., if $\acal$ forms a cone and sublevel sets of finitely many \fcts\ from $\acal$ generate the topology of $K$ (see Theorem 1' in \cite{Sic}). For \qpsh\ \fcts\ the condition on $\acal$ to be a cone is too strong, since \qpshy\ is not stable under addition. It turns out that the mentioned above condition can be relaxed so that the existence of the Shilov boundary for a wide class of \usc\ \fcts\ can be guaranteed. This will be the main part of the first chapter.

In the second chapter we define the closure of a subclass of \usc\ \fcts\ to be the collection of all limits of decreasing sequences of \fcts\ from that subclass. In our context it plays a role similar to the uniform closure of a subset of \cont\ \fcts\ on a compact Hausdorff space: The Shilov boundary for the subclass and its closure coincide.

The third chapter brings the Shilov boundary into connection to peak points. E.Bishop proved in \cite{Bi} that, if the compact Hausdorff space is assumed to be metrizable, then the closure of the set of peak points and the Shilov boundary for uniform subalgebras of \cont\ \fcts\ coincide. This is also true for any Banach subalgebra of \cont\ \fcts\ due to the results of H.G.Dales \cite{Dales} (see also \cite{Honary}). Note that using \usc\ \fcts\ similar identities were obtained in \cite{Sic} and \cite{Wittmann}. We apply these results to unions of uniform algebras and establish additional Bishop-type theorems.

In the fourth chapter we introduce the notions of \qpsh\ and \qhol\ \fcts\ and give a list of their properties. For the proofs and further results on these classes of \fcts\ we refer also to \cite{Sl}, \cite{Dieu} and \cite{TPESZ}.

In chapter five the results from the first three chapters are applied to subclasses of \qpsh\ \fcts.

In the sixth chapter Bychkov's theorem is generalized as follows: a boundary point of a convex bounded domain does not lie in the Shilov boundary for \qpsh\ or \qhol\ \fcts\ if and only if it is contained in an open part of a complex plane of dimension at least $q+1$ which is fully contained in the boundary of the given convex set.

It seems still to be an open question whether the Hausdorff dimension of Shilov boundary for holomorphic functions on compact sets in $\cbb^N$ is greater or equal to $N$. E.Bishop gave in \cite{Bi} a positive answer to this question in the special case $N=2$. In this context, we consider the Hausdorff dimension of the Shilov boundary for \qpsh\ \fcts\ on a convex bounded domain in chapter seven.

In chapter eight we show that the Shilov boundary for \qpsh\ and the Shilov boundary for $\ccal^2$-smooth \qpsh\ \fcts\ defined near a compact set coincide due to approximation techniques of Bungart \cite{Bu}, S\l odkowski \cite{Sl2} and Demailly \cite{Dem}. As an application we prove that if the given domain $D$ is bounded and smoothly bounded, then the Shilov boundary for \qpsh\ \fcts\ defined near $\cld$ is exactly the closure of the set of all strictly $q$-\psc\ points of the boundary of $D$. Using a rank condition on the Levi form of a defining \fct\ of $D$ which was established by M.Freeman in \cite{Free}, we obtain a foliation of parts of the Shilov boundary for \qpsh\ \fcts\ on $\cld$ by complex $q$-dimensional submanifolds.

\section{Shilov boundary for \usc\ \fcts}

In this chapter we will define the Bergman-Shilov boundary for subclasses of \usc\ \fcts\ and show its existence and uniqueness in certain cases. For the sake of abbreviation, we will simply talk about the Shilov boundary instead of the Bergman-Shilov boundary. Anyway, we have to point out that the concept of a distinguished boundary of certain domains in $\cbb^2$ was already introduced by S.Bergman in \cite{Berg}.

At first, we recall some basic definitions and facts about \usc\ \fcts\ on a compact Hausdorff space $K$.

\begin{defn} A \fct\ $f:K \to [-\infty,\infty)$ is called \textit{\usc}\ on $K$ if the sub-level set $\{ x \in K : f(x)<c\}$ is open in $K$ for every $c \in \rbb$. We denote then by $\USC(K)$ the set of all \usc\ \fcts\ on $K$ and by $\ccal(K)=\ccal(K,\cbb)$ the set of all complex-valued \cont\ \fcts\ on $K$.
\end{defn}

We will outline an important example for an \usc\ \fct.

\begin{ex} Let $S$ be a closed subset of $K$. Then the \textit{characteristic \fct\ $\chi_S$ of $S$ (in $K$)} given by
$$
\chi_S(x):=\left\{ \begin{array}{ll} 1, & x \in S \\ 0, & x \in K\setminus S \\ \end{array}\right.
$$
is \usc\ on $K$.
\end{ex}

The following statement is a well known fact.

\begin{lem}\label{lem-maxattain} Every \fct\ $f \in \USC(K)$ attains its maximum on $K$, i.e., there exists a point $x_0$ in $K$ such that
$$
\max\{ f(x) : x \in K\}:=f(x_0)=\sup\{ f(x) : x \in K\}.
$$
\end{lem}

From now on, $\acal$ is always a subset of $\USC(K)$. Our main object of study is the Shilov boundary for $\acal$.

\begin{defn}\label{defn-shilov} For a given \fct \ $f\in \USC(K)$ we set
$$
S(f):=\{x \in K: f(x)=\max_K f\}.
$$
A subset $S$ of $K$ is called a \textit{boundary for $\acal$} or \textit{$\acal$-boundary} if $S \cap S(f) \neq \emptyset$ for every $f \in \acal$. We denote by $\bA$ the set of all closed boundaries for $\acal$. The set $\SA:= \bigcap_{S \in \bA}S$ is called the \textit{Shilov boundary for $\acal$}.
\end{defn}

We give first some simple examples.

\begin{ex}\label{ex-shilov} \startenum

\new Let $f_1 = \chi_{\{0,1\}}$ and $f_2 =\chi_{\{1,2\}}$ considered as \usc\ \fcts\ on the interval $K=[0,2]$. For $\acal=\{f_1,f_2\}$ we have that $\{0,2\}, \{ 1\} \in \bA$, $S(f_1) \cap S(f_2) = \{ 1\}$ and that $\SA$ is empty.

\new For $f_1=\chi_{\{ 0\}}$ and $f_2=\chi_{\{ 1\}}$ considered as \fcts\ on $K=[0,1]$ we take $\acal=\{f_1,f_2\}$ and observe that $\{ 0,1\} \in \bA$, $S(f_1) \cap S(f_2) = \emptyset$ and $\SA=\{ 0,1\}$.

\new Consider the \fcts\ $f_1=\chi_{\{ -1,1\}}$ and $f_2=\chi_{{\{0\}}}$ defined on $[-1,1]$ and set $\acal=\{f_1,f_2\}$. Then $\{ -1,0\}, \{0,1\} \in \bA$, so $\SA=\{ 0\}$. But $\SA$ can not be an $\acal$-boundary because the \fct\ $f_1$ attains its maximum outside of zero.
\end{ex}

We have the following properties of Shilov boundaries.

\begin{prop}\label{prop-ashilov} \

\begin{enumerate}

\item The set $\SA$ is closed and possibly empty, whereas $\bA$ is never empty.

\item $S(f)$ is a closed non-empty subset of $K$.

\item If the set $\D T:= \bigcap_{f \in \acal} S(f)$ consists of more than two elements, then $\SA$ is empty.

\item If the set $T$ from above consists of one single element $x_0 \in K$ and $\SA \neq \emptyset$, then $\SA=\{ x_0\}$.

\item The set $S:=\bigcup_{f \in \acal} S(f)$ is an $\acal$-boundary.

\item If $\acal_1\subset\acal_2\subset \USC(K)$, then we have the following inclusions,
$$
b_{\acal_2} \subset b_{\acal_1} \qand \check{S}_{\acal_1} \subset \check{S}_{\acal_2}.
$$

\item Let $\acal=\bigcup_{j\in J}\acal_j$, where $\acal_j$ are subsets of $\USC(K)$. If $\check{S}_{\acal_j}$ are $\acal_j$-boundaries, then $\check{S}_\acal$ is an $\acal$-boundary and
$$
\SA = \overline{\bigcup_{j\in J} \check{S}_{\acal_j}}.
$$
\end{enumerate}
\end{prop}

\begin{proof}

\startenum

\new The set $\SA$ is closed as intersection of closed sets. Example \ref{ex-shilov} (1) shows that $\SA$ might be empty. The set $\bA$ contains at least the ambient space $K$.

\new Since $f\in\USC(K)$, the set $\{ x \in K : f(x)<\max_K f\}$ is open in $K$, so the set $S(f)=K\setminus\{ x \in K : f(x)<\max_K f\}$ is a closed subset of $K$. It is non-empty due to Lemma \ref{lem-maxattain}.

\new Pick two distinct elements $x_0,x_1$ from $T$. By definition $\{ x_0\}$ and $\{ x_1\}$ are $\acal$-boundaries and, thus, $\SA \subset \{ x_0\} \cap \{ x_1\} = \emptyset$.

\new In this case $\{ x_0 \} \in \bA$. Thus, $\emptyset \neq \SA \subset \{ x_0\}$ which yields $\SA=\{ x_0\}$.

\new The set $S$ is an $\acal$-boundary because $S \cap S(f)=S(f) \neq \emptyset$ for every $f \in \acal$.

\new This fact follows directly from definition.

\new The previous points (1) and (7) imply the inclusion $S:=\cl{\bigcup_{j\in J} \check{S}_{\acal_j}} \subset \SA$. By assumption, the set $S$ and, therefore, the set $\SA$ are non-empty.

Since an arbitrary \fct\ $f \in \acal$ is contained in $\acal_j$ for some $j \in J$ and by the assumption that $\check{S}_{\acal_j}$ is an $\acal_j$-boundary, we obtain that
$$
\emptyset \neq S(f) \cap \check{S}_{\acal_j} \subset  S(f) \cap S \subset  S(f) \cap \check{S}_{\acal}.
$$
This means that $S$ is an $\acal$-boundary and, thus, $\SA \subset S$. By the previous discussions above, we have that $S=\SA$ is an $\acal$-boundary.
\end{proof}

We can easily bring our concept of the Shilov boundary into relation with the classical Shilov boundary for uniform subalgebras of $\ccal(K)$.

\begin{rem}\label{rem-bshilov}

Let $\bcal$ be a subset of $\ccal(K)$. The classical Shilov boundary for $\bcal$ is the smallest closed subset $S$ of $K$ fulfilling $\max_S|f|=\max_K|f|$ for every $f \in \bcal$. Clearly, it corresponds to the Shilov boundary for the class $\log|\bcal|:=\{ \log|f|:f\in\bcal\}$. It then makes sense to simply write $b_\bcal$ and $\check{S}_\bcal$ instead of $b_{\log|\bcal|}$ and $\check{S}_{\log|\bcal|}$. It is clear that for the uniform closure $\cl{\bcal}$ of $\bcal$ in $\ccal(K)$ we have that $\check{S}_\bcal=\check{S}_{\cl{\bcal}}$.
\end{rem}

Now we recall the classical result of Shilov.

\begin{thmn}[Shilov] Let $K$ be a compact Hausdorff space and $\bcal$ a Banach subalgebra of $\ccal(K)$. Then $\check{S}_\bcal$ is non-empty and, moreover, it is a boundary for $\log|\bcal|$.
\end{thmn}

In this theorem the Banach algebra structure of $\bcal$ is heavily involved. We will extract the essential properties from that structure in order to establish similar results for Shilov boundaries for subclasses of \usc\ \fcts.

\begin{defn} Let $\acal$ be a subset of $\USC(K)$.
\begin{enumerate}

\item If $\acal_1$ and $\acal_2$ are two subfamilies of $\USC(K)$, then
$$
\acal_1+\acal_2:=\{ f+g:f\in\acal_1, g\in\acal_2\}.
$$

\item The family $\acal$ is a \textit{scalar cone} if $nf + b$ lies in $\acal$ for every $n \in \nbb_0=\nbb\cup\{0\}$, $f \in \acal$ and $b \in \rbb$. Here we use the convention $-\infty \cdot 0 = 0$.

\item The set $\acal$ is a \textit{cone} if $af+bg$ is contained in $\acal$ for every $a,b \in [0,+\infty)$ and $f,g \in \acal$.

\item An open set $V$ in $K$ is an \textit{$\acal$-polyhedron} if there exist finitely many \fcts\ $f_1,\ldots,f_n$ in $\acal$ and real numbers $C_1,\ldots,C_n$ such that
$$
V=V(f_1,\ldots,f_n)=\{x\in K : f_1(x)<C_1,\ldots, f_n(x)<C_n\}.
$$

\item The set $\acal$ \textit{generates the topology of $K$} if for every point $x\in K$ and every
\nbh\ $U$ of $x$ in $K$ there is an $\acal$-polyhedron $V$ such that $x\in V\subset U$.
\end{enumerate}
\end{defn}

Now we are able to show that the Shilov boundary for $\acal$ is a non-empty boundary for $\acal$ if $\acal$ possesses some simple structure. The following two statements are based on standard arguments used in the case of Banach subalgebras of \cont\ \fcts\ (see e.g. \cite{Wermer}, Theorem 9.1). First, we need the following lemma.

\begin{lem}\label{lem-shilov-bound} Let $\acal$ be a scalar cone. Assume that there exist an $\acal$-boundary $S \in \bA$ and an $\acal$-polyhedron $V{=}V(f_1,\ldots,f_n)$ such that $S \cap V=\emptyset$ and $\acal+\{f_j\}\subset\acal$ for $j=1,\ldots,n$. Given another $\acal$-boundary $E \in \bA$, it follows that $E\setminus V \in \bA$.
\end{lem}

\begin{proof} Since $\acal$ is a scalar cone and $\acal+\{f_j\}\subset\acal$ for $j=1,\ldots,n$, the constant \fct\ $0$ and, thus, $f_1,\ldots,f_n$ lie in $\acal$. Hence, we can assume that $V$ is of the form $V=\{x\in K : f_1(x){<}0,\ldots, f_n(x){<}0\}$.

Notice first that $E\setminus V$ is non-empty. Otherwise, $E \subset V$, so $\max_E f_j <0$ for every $j=1,\ldots,n$. Since $S$ does not meet $V$, there has to be an index $j_0\in \{1,\ldots,n \}$ such that $\max_S f_{j_0} \geq 0$. We obtain the contradiction $0 \leq \max_S f_{j_0}=\max_E f_{j_0} < 0$.

Suppose that the statement of the lemma is false, i.e., there are a point $y \in K$ and a \fct \ $f \in \acal$ such that $\max _{E\setminus V} f < \max_K f=f(y)$. Since $\acal$ is a scalar cone and $S \in \bA$, we can assume that $f(y)=0$ and $y \in S$. Consider for $m \in \nbb$ the functions $g_j:=mf + f_j \in \acal$, $j{=}1,{\ldots},n$. If $m$ is large enough, then $\max_{E\setminus V} g_j<0$ for each $j{=}1,\ldots,n$. Since $\max_K f =0$, it follows from the definition of $V$ that for every $j{=}1,\ldots,n$ we have that $g_j(x)<0$ for every $x \in V$. Hence, $\max_K g_j =\max_E g_j <0$ for every $j{=}1,\ldots,n$.

We conclude that $y \in V$. If not, there is an index $j_1 \in \{1,\ldots,n \}$ with $f_{j_1}(y)\geq 0$ and, thus, $g_{j_1}(y)\geq 0$, which is impossible. Thus, $y \in V\cap S = \emptyset$, a contradiction.
\end{proof}

%Proposition \ref{prop-shilov-bound}

\begin{thm}\label{thm-shilov-bound} If $\acal$ contains a subset $\acal_0$ which generates the topology of $K$ such that $\acal+\acal_0\subset\acal$, then the Shilov $\acal$-boundary is an $\acal$-boundary; i.e., $\SA \in \bA$.
\end{thm}

\begin{proof} At first, assume that $\acal$ is a scalar cone. If $\SA=K$, then there is nothing to show. So we can assume that $\SA\neq K$. We first treat the case $\SA \neq \emptyset$. Suppose $\SA \notin \bA$, then there is a \fct\ $f \in \acal$ such that $\max_{\SA} f < \max_K f$. Since $f$ is \usc\ on $K$, there is a \nbh\ $U$ of $\SA$ such that $f(x) < \max_K f$ for every $x \in U$. Then, since $\acal_0$ generates the topology of $K$, we conclude that for every $y \in L:=K\setminus U$ there are an $\acal_0$-polyhedron $V_y$ and an $\acal$-boundary $S_y \in \bA$ such that $y\in V_y$ and $V_y \cap S_y = \emptyset$. The family $\{ V_y\}_{y \in L}$ covers $L$. Hence, by the compactness of $L$, there are finitely many points $y_1,\ldots,y_l \in L$ such that the subfamily $\{ V_{y_j}\}_{j=1,\ldots,l}$ covers $L$. Since $\acal+\acal_0\subset\acal$, we can apply iteratively the previous Lemma \ref{lem-shilov-bound} in order to obtain that
$$
E:= \left(\left(\left( K\setminus V_{y_1}\right) \setminus V_{y_2}\right)\setminus \ldots \setminus V_{y_l}\right)= K\setminus \bigcup_{j=1}^l V_{y_j} \in \bA.
$$
Notice that, by the construction, the set $\SA$ lies in $E$ and, hence, $E$ is non-empty. Moreover, $E \subset U$ and, thus, $\max_E f < \max_K f$. But this contradicts to the fact that $E \in \bA$. Hence, $\SA \in \bA$.

In the case $\SA = \emptyset$, we pick an arbitrary point $p \in K$ and a \nbh\ $U$ of $p$ in $K$ which is an $\acal_0$-polyhedron of the form $U=\{ x \in K : f_1(x)<0,\ldots,f_k(x)<0\}$ such that $U \neq K$. Observe that for every $y \in K\setminus U$ there exists an $\acal$-boundary $S_y$ with $y \notin S_y$, since otherwise $y \in \SA$. Then we can choose an $\acal_0$-polyhedron $V_y$ such that $y \in V_y$, $p \notin V_y$ and $S_y \cap V_y$ is empty. By the same argument as above we can construct an $\acal$-boundary $E$ such that $p \in E \subset U$. But since $U \neq K$, there exists a point $x_0 \in K\setminus U$ and an index $k_0 \in \{ 1,\ldots,k\}$ such that $f_{k_0}(x_0)\geq 0$. This leads to the contradiction
$$
0 \leq f_{k_0}(x_0) \leq \max_K f_{k_0} = \max_E f_{k_0} < 0.
$$
Thus, $\SA$ can not be empty.

If $\acal$ is not necessarily a scalar cone, consider the set
$$
\tilde{\acal}:=\{ nf+c : n \in \nbb_0, \ f \in \acal, \ c \in \rbb\}.
$$
Since $\acal$ lies in $\tilde{\acal}$, we have that $b_{\tilde{\acal}} \subset b_{\acal}$ and $\cs_{\acal} \subset \cs_{\tilde{\acal}}$. Pick an arbitrary $\acal$-boundary $S$ and a \fct\ $nf+c \in \tilde{\acal}$, where $f \in \acal$, $n \in \nbb$ and $c \in \rbb$. Since $f$ and $nf+c$ attain their maximum at the same points, we have that
$$
S \cap S(nf+c)=S \cap S(f) \neq \emptyset.
$$
But this means that $S$ is also an $\tilde{\acal}$-boundary, so $b_\acal = b_{\tilde \acal}$ and $\cs_{\acal} =\cs_{\tilde{\acal}}$.

Now observe that the family $\tilde{\acal}_0:=\{ nf+c : n \in \nbb_0, \ f \in \acal_0, \ c \in \rbb\}$ generates the topology of $K$, since it contains $\acal_0$. Moreover, we have that $\tilde\acal+\tilde\acal_0\subset\tilde\acal$ and that $\tilde{\acal}$ is a scalar cone. Thus, by the previous discussions, we conclude that $\cs_{\acal}=\cs_{\tilde \acal} \in b_{\tilde \acal}=b_{\acal}$. This finishes the proof.
\end{proof} 

\section{Closure of a subfamily of \usc\ \fcts}

As in the previous section, $K$ will always be a compact Hausdorff space and $\acal$ a subfamily of \usc\ \fcts\ on $K$.

The limit of a decreasing sequence of \usc\ \fcts\ is again \usc. This simple fact will allow us to introduce the notion of the closure of $\acal$ and, hence, will give the possibility to compare the initial class with an approximating subclass.

\begin{defn}\label{def-cl-cones} The \textit{closure} $\scl{\acal}$ of $\acal$ is the set of pointwise limits of all decreasing sequences of \fcts\ in $\acal$. The set $\acal$ is \textit{closed} if $\scl{\acal}=\acal$.
\end{defn}

\begin{rem}\label{rem-cones} \startenum

\new If $\acal$ is a (scalar) cone, then its closure $\scl{\acal}$ is also a (scalar) cone.

\new The family $\acal$ generates the topology of $K$ if and only if $\scl{\acal}$ generates it. Indeed, one inclusion is trivial, since $\acal$ is always contained in $\scl{\acal}$. So let $\scl{\acal}$ generate the topology of $K$. Given a point $p \in K$ and an open \nbh\ $U$ of $p$ in $K$ there exists an $\scl{\acal}$-polyhedron $V=\{x \in K : f_1(x)<c_1,\ldots,f_k(x)<c_k\}$ such that $p \in V \subset U$. Since $f_1,\ldots f_k \in \scl{\acal}$, for every $j=1,\ldots,k$ there exists a sequence $(f_{j,n_j})_{n_j \in \nbb}$ of \fcts\ $f_{j,n_j} \in \acal$ which decreases to $f_j$ as $n_j$ tends to $\infty$. For large enough $n_0$ we have that $f_j(p)\leq f_{j,n_0}(p) < c_j$ for every $j=1,\ldots,k$. Then
$$
p \in V_0=\{ x \in K : f_{1,n_0}(x) < c_1,\ldots,f_{k,n_0}(x) < c_k\} \subset V \subset U
$$
and $V_0$ is an $\acal$-polyhedron. Thus, $\acal$ generates the topology of $K$.

\end{rem}

The notion of \textit{closure} introduced above has not the same meaning as the notion 'closure' in the topological sense, since in general it will not lead to a closed subclass of \usc\ \fcts. It becomes then an interesting question whether there is a better definition of the closure of $\acal$ which yields a closed set in our sense. Nevertheless, we will see later on that the notion introduced above is sufficient for our purposes.

\begin{ex} \startenum

\new Consider the following \usc\ \fcts\ on the compactification $K=[0,+\infty]$ of the interval $[0,+\infty)$. For an integer $n \in \nbb$ we set
$$
f_n:=\chi_{[1-\frac{1}{n+1},1]} \qand g_{n,k}:=1/k\cdot \chi_{\{ 1-\frac{1}{n+1}\}} + f_n.
$$
The \fcts\ $f_n$ decrease to $f_0:=\chi_{\{ 1\}}$. Now if $\acal$ is the set $\{ g_{k,n}: k,n \in \nbb\}$, then $\scl{\acal} = \acal \cup \{ f_n : n \in \nbb\}$ and $\scl{\left(\scl{\acal}\right)} = \scl{\acal} \cup \{ f_0\}$, but it is easy to see that $f_0$ can not be the limit of a decreasing sequence of \fcts\ from $\acal$.

\new One can think that after closing $\acal$ finitely many times we obtain a closed set. But this turns out to be wrong. Define for $k \in \nbb$ iteratively the $k$-th closure $\kscl{k}{\acal}$ of $\acal$ by $\scl{\left(\kscl{(k-1)}{\acal}\right)}$. Given $k\in \nbb$ and $n_0,\ldots,n_k \in \nbb$ consider the following \usc\ \fct\
$$
h_{n_0,\ldots,n_k}(x):=\sum_{j=0}^{k-1} g_{n_j n_{j+1}}(x-j),
$$
where $x \in [0,+\infty]$ and $g_{n_j n_{j+1}}$ are the \fcts\ from the previous example. We set $\acal:=\{ h_{n_0,\ldots,n_k}: k \in \nbb, n_0,\ldots, n_k \in \nbb\}$. Then we conclude that $\kscl{(k+1)}{\acal}$ contains the \fct\ $\chi_{\{1,\ldots,k \}}$, but not $\chi_{\{1,\ldots,k+1 \}}$.

\new Even if we take the union of all $l$-th closures it will not lead to a closed set. Consider now for given integers $k \in \nbb$ and $n_0,\ldots,n_k \in \nbb$ the \fcts\
$$
G_k:=\chi_{\{\infty\}}+ \sum_{j=k+1}^\infty(1+1/j)\chi_{\{ j\}} \qand H_{n_0,\ldots,n_k} := h_{n_0,\ldots,n_k} + G_k,
$$
where $h_{n_0,\ldots,n_k}$ are the \fcts\ from the example above. Now consider the family $\acal:=\{ H_{n_0,\ldots,n_k}: k \in \nbb, n_0,\ldots, n_k \in \nbb\}$. Then by the same argument as before we can derive that $\bigcup_{l\in \nbb} \kscl{l}{\acal}$ contains $\chi_{\{1,\ldots,k\}} +g_k$ for every $k \in \nbb$, but it does not contain $\chi_{\{1,2,\ldots,\infty\}}$. Anyway, the \fcts\ $\chi_{\{1,\ldots,k\}} +g_k$ decrease to $\chi_{\{1,2,\ldots,\infty\}}$ as $k$ tends to $\infty$. Hence,
$$
\chi_{\{1,2,\ldots,\infty\}} \in \scl{\left(\bigcup_{l\in \nbb} \kscl{l}{\acal} \right)},
$$
but $\chi_{\{1,2,\ldots,\infty\}} \notin \kscl{l}{\acal}$ for every $l \in \nbb$.
\end{ex}

We have seen by the previous examples that each iterate closure of $\acal$ might lead to a larger set. Nevertheless, this additional \fcts\ will not contribute to the Shilov boundary in the following sense.

\begin{lem}\label{lem-decr-cont} Let $f$ be \usc\ on $K$ and $(f_n)_{n \in \nbb}$ a sequence of \usc\ \fcts\ decreasing to $f$. Assume that $f$ is bounded above by a \fct\ $g$ which is \lsc\ on $K$, i.e., $f < g$ on $K$. Then there is an index $n_0$ such that $f_n < g$ on $K$ for every $n \geq n_0$.
\end{lem}

\begin{proof} Take a point $x \in K$. Then there is an index $n_x\in \nbb$ such that $f(x) \leq f_{n_x}(x) < g(x)$. Since $f_{n_x}-g$ is \usc\ on $K$, we can find an open \nbh\ $U_x$ of $x$ in $K$ such that $f_{n_x}(y)<g(y)$ for every $y \in U_x$. By compactness of $K$ we can cover $K$ by finitely many open sets $U_{x_1},\ldots,U_{x_l}$ from the covering $\{ U_x\}_{x \in K}$. We set $n_0:=\max\{ n_{x_j} : j = 1,\ldots,l\}$. Since $(f_n)_{n\in\nbb}$ is decreasing, we obtain that $f_n \leq f_{n_0} < g$ on $K$ for every $n\geq n_0$.
\end{proof}

\begin{lem}\label{lem-decr-sh} Let $(f_n)_{n \in \nbb}$ be a sequence of \usc\ \fcts\ on $K$ decreasing to $f$. Then
$\lim_{n \to \infty} \max_K f_n = \max_K f$.
\end{lem}

\begin{proof} The limit $a:=\lim_{n \to \infty} \max_K f_n$ exists because $(\max_K f_n)_{n \in \nbb}$ is a decreasing sequence bounded below by $\max_K f$. Assume that $a > \max_K f$. By the previous Lemma \ref{lem-decr-cont} we can find a large enough integer $n_0$ such that $a > f_{n_0}(y)$ for every $y \in K$, which is a contradiction to the definition of $a$.
\end{proof}

\begin{cor}\label{cor-shilov-cl} The set of all $\acal$-boundaries coincides with the set of all $\scl{\acal}$-boundaries, i.e.,
$$
b_{\acal}=b_{\scl{\acal}} \qand \SA=\check{S}_{\scl{\acal}}.
$$
\end{cor} 

\section{Minimal boundary and peak points}

In this section we discuss the relation between the Shilov boundary and peak points based on the main result of Bishop in \cite{Bi}. As before, let $\acal$ always be a subfamily of \usc\ \fcts\ on a compact Hausdorff space $K$.

\begin{defn}  \

\benum

\item We denote by $\BA$ the set of all (possibly non-closed) boundaries for $\acal$ (recall Definition \ref{defn-shilov}).

\item If there exists a subset $\MA$ in $\BA$ such that $\MA$ is contained in every boundary for $\acal$, then this set will be called the \textit{minimal boundary for $\acal$}.

\item A point $x\in K$ is called \textit{peak point for $\acal$} if there is a \fct\ $f \in \acal$ such that $S(f)=\{ x\}$. We say that $f$ \textit{peaks at $x$}. We denote by $\PA$ the set of all \textit{peak points for $\acal$}.

\eenum

\end{defn}

The sets $\MA$, $\PA$ and $\SA$ are possibly empty. If $\MA$ is non-empty, it is not necessarily closed, while $\SA$ is by definition always a closed subset of $K$. The following examples show that the sets $\MA$, $\SA$ and $\PA$ may differ or might be empty.

\begin{ex}\label{ex-pms} \startenum

\new We enumerate the subset $L=[0,1] \cap \qbb$ of $K=[0,1]$ by a sequence $(x_n)_{n\in\nbb}$. For the subclass $\acal=\{ \chi_{\{x_n\}}: n \in \nbb\}$ of \usc\ \fcts\ on $K$, we have that $\PA=\MA=L \subsetneq \SA=[0,1]$.

\new There exists a separating Banach algebra of \cont\ \fcts\ on a compact set with no minimal boundary.

There exists a Banach algebra of \cont\ \fcts\ on a compact set such that the minimal boundary is an open non-empty set. For both examples we refer to \cite{Bi}.

\new By Example \ref{ex-shilov} (1) we can see that there is a subclass $\acal$ of $\USC(K)$ such that $\SA$, $\PA$ and $\MA$ are all empty.

\end{ex}

\begin{rem}

Given $f \in \USC(K)$ define $\acal':=\acal \cup \{ f\}$. The existence of $m_\acal$ does not imply the existence of $m_{\acal'}$ in general. To see this consider $\acal=\{ \chi_{\{ 0\}}\}$ and $f=\chi_{\{ -1,1\}}$ on the interval $[-1,1]$. Even though $m_\acal=\{ 0\}$, $m_{\acal'}$ does not exist. On the other hand it is easy to verify that, if we choose $f$ which peaks at some point $x \in K$, then $m_{\acal'}=m_{\acal} \cup \{ x\}$.
\end{rem}

We give some properties and relations between the above defined sets.

\begin{prop}\label{prop-pms} \

\benum

\item The set $P_\acal$ lies in every $\acal$-boundary $S$ from $B_\acal$. If $P_\acal$ is itself an $\acal$-boundary, then it is exactly the minimal boundary $m_\acal$.

\item The inclusions $\PA \subset \MA \subset \SA$ hold whenever $\MA$ exists.

\item $\SA=\cl{m}_\acal$, if $\MA$ exists.

\item Let $\acal_1\subset\acal_2\subset \USC(K)$. Then we have the following inclusions,
$$
B_{\acal_2} \subset B_{\acal_1} \qand P_{\acal_1} \subset P_{\acal_2}.
$$
If $m_{\acal_1}$ and $m_{\acal_2}$ exist, then $m_{\acal_1}\subset m_{\acal_2}$.

\item Let $\acal=\bigcup_{j\in J}\acal_j$, where $\acal_j$ are subsets of $\USC(K)$. Then $\PA = \bigcup_{j\in J} P_{\acal_j}$. If $m_{\acal_j}$ exists for every $j \in J$, then $\MA$ exists and $\MA = \bigcup_{j\in J} m_{\acal_j}$.

\eenum
\end{prop}

\begin{proof} \startenum

\new Let $x \in \PA$ and $f \in \acal$ such that $f$ peaks at $x$. Given an $\acal$-boundary $S$, it is clear that $S \cap S(f)=\{ x\}$. In particular, the point $x$ lies in $S$. This yields the inclusion $\PA \subset S$. Now if $\PA$ lies in $\BA$, then by the previous discussion and by the definition of the minimal boundary for $\acal$, we have that $\PA=\MA$.

\new  Since $\MA \in \BA$ and by the previous property (1), we obtain that $\PA \subset \MA$. Since $\MA$ is the smallest $\acal$-boundary, $\MA \subset S$ for every $S \in \bA$. Hence, $\MA \subset \SA$.

\new Since $\SA$ is closed and $\MA \subset \SA$ by the previous point (2), $\cl{m}_\acal$ is a subset of $\SA$. On the other hand, $\MA$ is contained in $\BA$, and therefore $\cl{m}_\acal$ is a closed $\acal$-boundary. By definition this means that $\SA\subset\cl{m}_\acal$. Hence, $\cl{m}_\acal=\SA$.

\new This inclusions follow immediately from the definitions of the corresponding sets.

\new The identity $\PA = \bigcup_{j\in J} P_{\acal_j}$ is obvious. We show that $m:=\bigcup_{j\in J} m_{\acal_j}$ is a minimal $\acal$-boundary. Pick an arbitrary \fct\ $f \in \acal$. Then $f \in \acal_j$ for some index $j \in J$. By assumption $m_{\acal_j}$ is a minimal boundary for $\acal_j$. Thus, we obtain that
$$
\emptyset \neq S(f) \cap m_{\acal_j} \subset S(f) \cap m.
$$
But this implies that $m \in \BA$. Now let $S$ be an arbitrary $\acal$-boundary. By point (4) we have that $S \in \BA \subset B_{\acal_j}$ for every $j \in J$. Then $m_{\acal_j} \subset S$ for all $j \in J$ and, thus, $m \subset S$. This shows the minimality of $m$, so $\MA=m$.
\end{proof}

In what follows, we present some Bishop-type theorems for subclasses of \usc\ \fcts\ on a metrizable compact space $K$.

\begin{defn} \

\benum

\item A topological space $K$ is \textit{metrizable} if it has a metric which induces the given topology. In this case its topology admits a countable base.

\item A subset $\bcal$ of $\ccal(K)$ is \textit{separating} or \textit{separating points of $K$} if for every $x,y \in K$ there exists a \fct\ $f \in \bcal$ such that $f(x)\neq f(y)$.

\item Given a subclass $\acal$ of $\USC(K)$, it is \textit{strictly separating} or \textit{strictly separating points of $K$} if for every $x,y \in K$ there exist \fcts\ $f_1,f_2 \in \acal$ such that $f_1(x)>f_1(y)$ and $f_2(x)<f_2(y)$.

\eenum

\end{defn}

We recall Bishop's theorem. Further generalizations can be found in \cite{Sic}, \cite{Dales} and \cite{Honary}.

\bdefn
For $\bcal$ being a subset of $\ccal(K)$ we use the same simplification of notations as in Remark \ref{rem-bshilov} above. Namely, we write $B_{\bcal}$, $m_{\bcal}$ and $P_{\bcal}$ instead of $B_{\log|\bcal|}$, $m_{\log|\bcal|}$ and $P_{\log|\bcal|}$, respectively.
\edefn

\begin{thmn}[Bishop, \cite{Bi}] Let $K$ be a compact metrizable Hausdorff space and $\bcal$ a separating uniform subalgebra of $\ccal(K)$. Then the minimal boundary of $\bcal$ exists and is exactly the set of all peak points for $\bcal$.
\end{thmn}

\begin{cor}\label{bishop-unhol} Suppose $\bcal$ is a union of separating uniform subalgebras $(\bcal_j)_{j \in J}$ of $\ccal(K)$, where $K$ is a metrizable compact Hausdorff space. Then $m_\bcal$ exists and
$$
m_\bcal = P_\bcal \qand \check{S}_\bcal=\cl{P_\bcal}.
$$
\end{cor}

\begin{proof} By Bishop's theorem $m_{\bcal_j}$ exists and coincides with $P_{\acal_j}$ for every $j \in J$. By Proposition \ref{prop-pms} (5), we obtain that $m_\bcal$ is the minimal boundary for $\bcal$ and
$$
P_{\bcal} = \bigcup_{j \in J} P_{\bcal_j} = \bigcup_{j \in J} m_{\bcal_j} = m_\bcal.
$$
The identity $\check{S}_\bcal=\cl{P}_\bcal$ follows now from Proposition \ref{prop-pms} (3).
\end{proof}

For closed cones of \usc\ \fcts, i.e., for cones $\acal$ having the property $\scl{\acal}=\acal$, we are able to obtain another Bishop type theorem. The proof is nearly the same as in Theorem 1 in \cite{Bi}. A similar result for subfamilies of non-negative \cont\ \fcts\ was already obtained by Siciak in \cite{Sic} (see Theorem 3).

\begin{thm}\label{shilov-peak} Let $K$ be a metrizable compact Hausdorff space. Let $\acal$ be a closed cone in $\USC(K)$ containing real constants and strictly separating points of $K$. Then $m_\acal$ exists and coincides with the set of all peak points for $\acal$; i.e.,
$$
\MA = \PA \quad \mathrm{and} \quad \SA = \cl{P}_\acal.
$$
\end{thm}

\begin{proof} In view of Proposition \ref{prop-pms} (1) and (2), we only need to show that $P_\acal$ is a non-empty $\acal$-boundary or, equivalently, $\PA\cap S(f) \neq \emptyset$ for every $f \in \acal$.

Fix a \fct\ $f\in\acal$. Denote by $\Gamma$ the set of all peak sets for $\acal$, i.e.,
$$
\Gamma:=\{ \gamma \subset K : \exists \ f_\gamma \in \acal \ \mathrm{such \ that} \ S(f_\gamma)=\gamma\}.
$$
Let $\mathfrak{S}$ be the set of subsets $\tilde{\Gamma}$ of $\Gamma$ which contain $S(f)$ and which have the \textit{finite intersection property (fip)}, i.e., for every finite family $\{ \gamma_i\}_{i\in I}$ of elements in $\tilde{\Gamma}$ its intersection $\bigcap_{i\in I} \gamma_i$ is non-empty. Let $(\tilde{\Gamma}_j)_{j \in J}$ be a totally ordered set in $\mathfrak{S}$. We infer that $\ucal:=\bigcup_{j\in J} \tilde{\Gamma}_j$ is an upper bound for elements in $(\tilde{\Gamma}_j)_{j \in J}$ and that it is contained in $\mathfrak{S}$. Indeed, it is obvious that $\ucal$ bounds all elements of $(\tilde{\Gamma}_j)_{j\in J}$ and that $S(f)$ lies in $\ucal$. Let $\{ \gamma_i\}_{\in I}$ be a finite family in $\ucal$. Since $(\tilde{\Gamma}_j)_{j \in J}$ is totally ordered, there is an index $j_0 \in J$ such that $\gamma_i \in \tilde{\Gamma}_{j_0}$ for every $i \in I$. But $\tilde{\Gamma}_{j_0}$ has the (fip). Therefore, $\bigcap_{i\in I} \gamma_i$ is non-empty. This implies that $\ucal$ has the (fip) and, thus, it is contained in $\mathfrak{S}$. By Zorn's Lemma $\mathfrak{S}$ has a maximal element $\Gamma_0$. It contains $S(f)$, has the (fip) and no larger subset of $\Gamma$ has the (fip).

Since all the sets in $\Gamma_0$ are closed (see Proposition \ref{prop-ashilov} (2)), the set $K$ is compact and $\Gamma_0$ has the (fip), it follows that the set $D:=\bigcap_{\gamma \in \Gamma_0} \gamma$ is closed and non-empty. The set $K\setminus D$ is covered by open sets $K\setminus\gamma$, where $\gamma \in \Gamma_0$. Since $K$ and, therefore, also $K\setminus D$ are metrizable, we can choose a countable sequence $(\gamma_n)_{n\in\nbb}$ of sets $\gamma_n$ in $\Gamma_0$ such that $D=\bigcap_{n \in \nbb}\gamma_n$.

Let $x_0$ be a point in $D$. Then the \fcts\ $f_n:=f_{\gamma_n}-f_{\gamma_n}(x_0)$ fulfill $f_n\in\acal$, $S(f_n)=S(f_{\gamma_n})=\gamma_n$, $f_n\leq 0$ on $K$ and $f_n(x_0)=0$.

Define $g$ as the \fct\ $g:=\sum_{n=1}^\infty f_n$. Since $g$ is a limit of a decreasing sequence of \fcts\ $\sum_{n=1}^k f_n$ in $\acal$, and since $\acal$ is closed, we deduce that $g(x_0)=0$, $g\leq 0$ on $K$ and that $g$ is contained in $\acal$. In addition, if $x\in K \setminus \gamma_n$, then $f_n(x)<0$. This implies that $g$ can not attain a maximum on $K$ in $x$ because $g(x)<0$ and $g(x_0)=0$. Therefore, $g$ can only attain maximal values on $K$ in $\gamma_n$, i.e., $S(g)\subset \gamma_n$. Since this inclusion holds for all $n\in\nbb$, we have that $S(g)$ is a subset of $D$.

We claim that $S(g)$ contains only a single point, namely $x_0$. If this is true, $x_0$ is a peak point for $\acal$. Then by the construction of $g$ and the definition of $D$ we have that $x_0 \in S(g)\subset D \subset S(f)$. But this means that $x_0\in \PA\cap S(f)$. In particular, the intersection $\PA\cap S(f)$ is non-empty. Since $f$ is an arbitrary function in $\acal$, it follows from the definition that $\PA$ is an $\acal$-boundary which implies that $\MA \subset \PA$.

It remains to proof the claim above that $S(g)$ consists only of a single point. Assume that this is false so that $S(g)$ contains more than a single point. Since $\acal$ separates the points of $K$, there is a \fct\ $h$ in $\acal$ and a point $x_1 \in S(g)$ such that $\D h(x_1)=\max_{S(g)} h=0$ and $h$ is not constant on $S(g)$. Then the set $E:=\{ x \in S(g): h(x)=0\}$ is a proper closed subset of $S(g)$ containing $x_1$.

Consider the sets $T_n:=\{ x \in K : h(x)\geq 1/n\}$, where $n \in \nbb$. These sets are closed and disjoint from $S(g)$. This means that $g<0$ on $T_n$ for every $n\in\nbb$. Hence, for each $n\in\nbb$ we can choose a large enough constant $c_n\in\nbb$ such that $\D \max_{T_n}\{h+c_n g\}<0$. Recall that $g$ is non-negative on $K$.

Define the \fct\ $\vphi$ by $\vphi:=h+\sum_{n=1}^\infty c_n g$. Since $\vphi$ is the limit of a decreasing sequence of \fcts\ $h+\sum_{n=1}^k c_n g \in \acal$, it also lies in $\acal$. For a point $x \in K\setminus \bigcup_n T_n$ we have that $h(x)\leq 0$ and $g(x)\leq 0$ and therefore $\vphi(x)\leq 0$. In addition, on $T_j$ it holds that
$$
\vphi = h + \sum_{n=1}^\infty c_n g \leq \max_{T_j}\{ h + c_j g\}+ \sum_{n=1,n\neq j}^\infty c_j g <0.
$$
This implies that $x_1 \in S(\vphi)$ because $\vphi(x_1)=0$ and $\vphi \leq 0$ on $K$. Moreover, we have that $S(\vphi)\in\Gamma$ and  $x_1\in S(g)\cap S(\vphi)$. We assert that $S(\vphi) \in \Gamma_0$. To show this, we set $\Gamma_1:=\Gamma_0 \cup \{ S(\vphi)\}$. Since $S(g) \subset D$, we have that
$$
\emptyset \neq S(g) \cap S(\vphi) \subset S(\vphi) \cap D \subset S(\vphi) \cap \gamma
$$
for every $\gamma \in \Gamma_0$. Together with the (fip) of $\Gamma_0$ it follows that $\Gamma_1 \in \mathfrak{S}$. Since $\Gamma_0 \subset \Gamma_1$ and by the maximality of $\Gamma_0$, we conclude that $\Gamma_0=\Gamma_1$ and, thus, $S(\vphi) \in \Gamma_0$.

Therefore, $S(g)\subset D \subset S(\vphi)$, since $D=\bigcap_{\gamma\in\Gamma_0}\gamma$. Recall that $S(g)\setminus E\neq \emptyset$. Now if $x \in S(g)\setminus E$, we have that $g(x)=0$ but $h(x)<0$ and thus $\vphi(x)<0$. Hence, $S(g)\setminus E$ and $S(\vphi)$ are disjoint, which is a contradiction to $S(g) \subset S(\vphi)$. Finally, we have shown that $m_\acal$ is contained in $P_\acal$.
\end{proof}

Since the families of \fcts\ we will define later do not form cones, we need another peak point theorem.

\begin{defn} Let $\acal$ be a subclass of \usc\ \fcts\ on $K$ and let $\Theta$ be a subset of non-negative \cont\ \fcts\ on $K$ with the following property: for each $x \in K$ and each closed subset $S$ of $K$ with $x \notin S$ there exists a \fct\ $\vartheta \in \Theta$ such that $S(\vartheta)=\{ x\}$ and $\vartheta$ vanishes on $S$. We say that a \fct\ $f \in \acal$ is a \textit{strictly-$\acal$-\fct\ with respect to $\Theta$} if for every $\vartheta \in \Theta$ there is a number $\eps_0>0$ such that $f+\eps \vartheta \in \acal$ for every $\eps \in (-\eps_0,\eps_0)$. The subfamily of $\acal$ consisting of all strictly-$\acal$-\fcts\ with respect to $\Theta$ is denoted by $\acal[\Theta]$.
\end{defn}

\begin{thm}\label{bishop-sa} Let $\acal$ be a subclass of \usc\ \fcts\ on $K$. Suppose that there exist a subclass $\Theta$ as in the previous definition and a positive \fct\ $\omega \in \acal[\Theta]$ such that $\acal + \{\eps\omega\} \in \acal[\Theta]$ for every positive number $\eps>0$. Then $\check{S}_{\acal}=\cl{P}_{\acal}\in \bA$.
\end{thm}

\begin{proof} First, observe that $P_{\acal[\Theta]}$ is non-empty. Indeed, the \fct\ $\omega$ attains its maximum on $K$, say at a point $x_0 \in K$. Pick a \fct\ $\vartheta \in \Theta$ with $S(\vartheta)=\{ x_0\}$. Then there is a positive number $\delta>0$ such that $\omega+\delta\vartheta$ lies in $\acal$. But then $2\omega+\delta\vartheta$ is in $\acal[\Theta]$ by the assumption made on $\omega$. Moreover, $S(2\omega+\delta\vartheta)=\{ x_0\}$ and, thus, $x_0 \in P_{\acal[\Theta]}$.

The set $P_{\acal[\Theta]}$ is a subset of $\check{S}_{\acal[\Theta]}$. We show that $S:=\cl{P}_{\acal[\Theta]}$ is a boundary for the class $\acal[\Theta]$. If not, there exists a \fct\ $f \in \acal[\Theta]$ such that $\max_K f > \max_{S} f$. For a small enough number $\eps_0>0$ we have that $\max_K g > \max_{S} g$, where $g:=f+\eps_0 \omega$. Then there exists a point $x_1 \in K\setminus S$ such that $g(x_1)=\max_{K} g$. Let $\theta$ be a \fct\ from $\Theta$ such that $S(\theta)=\{ x_1\}$ and $\theta$ vanishes on $S$. In particular, $\theta(x_1)>0$. Then for a small enough number $\eps_1>0$ the \fct\ $f+\eps_1\theta$ is in $\acal$. Hence, the \fct\ $h:=g + \eps_1\theta=f+\eps_1\theta+\eps_0 \omega$ lies in $\acal[\Theta]$ and fulfills $S(h)=\{ x_1\}$. Thus, $x_1 \in P_{\acal[\Theta]} \subset S$. But this contradicts to
$$
\max_{S} h = \max_{S} g < \max_K g = g(x_1) < h(x_1) \leq \max_{S} h.
$$
Therefore, $\check{S}_{\acal[\Theta]}$ is contained in $S$. Altogether, we have that $\cl{P}_{\acal[\Theta]}=\check{S}_{\acal[\Theta]}$.

Let $f$ be an arbitrary \fct\ from $\acal$. Then the sequence $(f_n)_{n\in\nbb}$ of \fcts\ $f_n:=f+(1/n)\omega$ in $\acal[\Theta]$ decreases to $f$. This implies that $\acal$ lies in the closure of $\acal[\Theta]$. Since $\acal[\Theta]$ lies in $\acal$ and in view of  Corollary \ref{cor-shilov-cl}, we have that $\bA = b_{\acal[\Theta]}$ and $\check{S}_{\acal[\Theta]}=\check{S}_\acal$. Finally, the proof is done due to the following inclusions,
$$
\SA = \check{S}_{\acal[\Theta]} = \cl{P}_{\acal[\Theta]} \subset \cl{P}_\acal \subset \SA.
$$
\end{proof}

\section{\qpsh\ and \qhol\ \fcts}

We give a short overview of what is known about \qpsh\ and \qhol\ \fcts.

\begin{defn}\label{def-qpsh} Let $U$ be an open set in $\cbb^N$ and $u:U\to[-\infty,\infty)$ be an \usc\ \fct\ on $U$, i.e., $\{ z \in U : u(z)<c\}$ is open for every $c \in \rbb$. Let $q\in\{0,1,\ldots,N{-}1\}$.

\benum

\item The \fct\ $u$ is called \textit{\sph}\ on $U$ if for every ball $B\relc U$ and every \fct\ $h$ which is
\ph\ in a \nbh\ of the closure of $B$ and fulfills $u\leq h$ on $bB$ one has that $u\leq h$ on $B$.

\item The \fct\ $u$ is called \textit{\qpsh}\ in $U$ if it is \sph\ in $U\cap\pi$ for every $(q{+}1)$-dimensional complex affine plane $\pi\subset\cbb^N$.

\item By $\PSH_q(U)$ we denote the set of all \qpsh\ \fcts\ on $U$. If $q$ is an integer with $q\geq N$, we simply define $\PSH_q(U):=\USC(U)$.

\item Given a compact set $K$ in $\cbb^N$ the set $\PSH_q(K)$ denotes the set of all \fcts\ $v \in \USC(K)$ which have an \qpsh\ extension into an open \nbh\ of $K$, i.e., there exists an open \nbh\ $U$ of $K$ and a \fct\ $u \in \PSH_q(U)$ such that $u|K=v$.

\item We set $\acal\PSH_q^0(K):=\PSH_q(\inti{K}) \cap \ccal(K)$.

\eenum
\end{defn}

We give a overview of the basic properties of \qpsh\ \fcts\ although we might not use them explicitly in the following sections. We refer to \cite{Dieu}, \cite{HM}, \cite{Sl}, \cite{Sl2} and \cite{TPESZ} for details and further properties.

\begin{prop}\label{prop-qpsh} Let $U$ be an open set in $\cbb^N$.

\benum

\item The 0-\psh\ \fcts\ are the classical \psh\ \fcts.

\item $\PSH_0(U) \subset \PSH_1(U) \subset \ldots \subset \PSH_{N-1}(U) \subset \USC(U)$

\item Given $c\geq 0$ and two functions $u\in\PSH_q(U)$ and $v\in\PSH_r(U)$,
$$\begin{array}{ll}
cu\in\PSH_q(U),&\max\{u,v\}\in\PSH_{\max\{q,r\}}(U),\\
u+v\in\PSH_{q+r}(U),&\min\{u,v\}\in\PSH_{q+r+1}(U).
\end{array}$$

\item \label{qpsh-smooth} The $\ccal^2$-smooth function $u$ lies in $\PSH_q(U)$ if and only if its
complex Hessian $(u_{z_k \bar{z}_l})_{k,l=1\ldots,N}$ has at least $N{-}q$ non-negative eigenvalues at each point in $U$.

\item Let $u_j$ be $q_j$-\psh\ \fcts, $j=1,\ldots,k$. If $\chi:\rbb^k \to \rbb$ is a $\ccal^2$-smooth convex \fct\ which is non-decreasing in each variable, then the composition $\chi(u_1,\ldots,u_k)$ is a $\tilde{q}$-\psh\ \fct\ with $\tilde{q}=q_1+\ldots+q_k$.

\item \label{qpsh-hol} If $\psi \in \PSH_q(U)$, then $\psi \circ h \in \PSH_q(W)$ for every holomorphic mapping $h:W \to U$, where $W$ is an open set in $\cbb^n$.

\item \label{qpsh-decreasing} If $(u_n)_{n\in\nbb}$ is a decreasing sequence of functions in $\PSH_q(U)$, then the
limit $\D\lim_{n\to\infty}u_n$ lies in $\PSH_q(U)$.

\item \label{qpsh-supremum} Let $\{u_j\}_{j\in{J}}$ be a locally bounded family of \fcts\ in $\PSH_q(U)$. Then $u^\star=\left(\sup_j u_j\right)^\star$ lies in $\PSH_q(U)$. Here, $u^\star$ means the \textit{\usc\ regularization of $u$}, i.e., $u^\star(z):=\limsup_{\zeta \to z, \zeta \in U}u(\zeta)$ for $z \in U$.

\item \label{qpsh-patch} Let $V$ be an open subset of $U$. Let $v \in \PSH_q(V)$ and $u \in \PSH_q(U)$ such that $\limsup_{\zeta \to z, \zeta \in V}v(\zeta)\leq u(z)$ for every $z \in U\cap bV$. Then we have that
    $$
    \vphi=\left\{\begin{array}{ll}
    u & \mathrm{on} \ U\setminus V\\
    \max\{u,v\} & \mathrm{on} \ V
    \end{array}\right\}\in \PSH_q(U).
    $$

\item \label{qpsh-maxpr} Let $q{<}N$ and $U \relc \cbb^N$. Then every \fct\ $u\in\PSH_q(U)\cap \USC(\cl{U})$ satisfies the
maximum principle, i.e.,
$$
\max_{\overline{U}} u = \max_{bU} u.
$$
\eenum
\end{prop}

\begin{proof} The properties (1) and (2) follow directly from the definition. Regarding (3), it is not hard to verify that $cu\in\PSH_q(U)$ and $\max\{u,v\}\in\PSH_{\max\{q,r\}}(U)$. The proofs of the properties $u+v\in\PSH_{q+r}(U)$ and $\min\{u,v\}\in\PSH_{q+r+1}(U)$ can be found in \cite{Sl2}. For the proofs of (4), (9) and (10) we refer to \cite{HM}. The proofs of (5) and (6) can be found in, e.g., \cite{TPESZ}. The properties (7) and (8) are easy to verify.
\end{proof}

A generalization of holomorphic functions is given by the so called \qhol\ \fcts\ which were already studied by, e.g., Basener in \cite{Ba} and \cite{Ba2} and Hunt and Murray in \cite{HM}.

\begin{defn}\label{def-qhol} Let $U$ be an open set in $\cbb^N$.

\benum

\item Given an integer $q\geq 0$, the set of \textit{$q$-holomorphic} functions on $U$ is defined by
$$
\ocal_q(U):=\{ f \in \ccal^2(U) : \dbar f \wedge (\dd\dbar f)^q=0\}.
$$

\item Let $K$ be a compact set in $\cbb^N$. The set $\ocal_q(K)$ denotes the set of all \cont\ \fcts\ $f$ on $K$ which have a \qhol\ extension into an open \nbh\ of $K$, i.e., there exist an open \nbh\ $U$ of $K$ and a \fct\ $F \in \ocal_q(U)$ such that $F|K=f$.

\item We define $A_q(K):=\ocal_q(\inti{K}) \cap \ccal(K)$.

\eenum
\end{defn}

The next proposition is a collection of properties of \qhol\ \fcts.

\begin{prop}\label{prop-qhol}  Let again $U$ be an open set in $\cbb^N$.
\benum

\item The $0$-holomorphic functions are the usual holomorphic functions.

\item If $q\geq N$, then $\ocal_q(U)=\ccal^2(U)$.

\item $\ocal(U)\subset \ocal_1(U)\subset \ldots \subset \ocal_{N-1}(U) \subset \ocal_N(U)=\ccal^2(U)$

\item A function $f\in\ccal^2(U)$ lies in $\ocal_q(U)$ if and only if
$$
\mathrm{rank}\left(\begin{matrix}
f_{\bar{z}_1} & \cdots & f_{\bar{z}_N}\\
f_{z_1\bar{z}_1} & \cdots & f_{z_1\bar{z}_N}\\
\vdots&\ddots&\vdots\\
f_{z_N\bar{z}_1}&\cdots&f_{z_N\bar{z}_N}
\end{matrix}\right)\leq{q}\quad\mathrm{on} \ U.
$$

\item If $f \in \ocal_q(U)$, $g \in \ocal_r(U)$, $\lambda \in \cbb$ and $m\in \nbb$, then
$$
f^m, \lambda f \in \ocal_q(U) \qand fg, f+g \in \ocal_{q+r}(U).
$$

\item If $W$ is an open set in $\cbb^n$, $f \in \ocal_q(W)$ and $h:U \nach W$ a holomorphic mapping, then $f\circ h \in \ocal_q(U)$.

\item If $f\in\ocal_q(U)$ and $h$ is a complex valued
    holomorphic function defined in the \nbh\ of the image of $f$, then $h \circ f \in \ocal_q(U)$.

\item \label{loc-max-qhol} If $q<N$, then every \fct\ $f\in \ocal_q(U)$ admits \textit{the local maximum modulus principle}, i.e., for every relatively compact set $D\relc U$ we have that
$$
\max_{\cl{D}}|f|=\max_{bD} |f|.
$$

\item If $f \in \ocal_q(U)$, then $\Re f$ and $\log|f|$ lie in $\PSH_q(U)$.

\eenum

\end{prop}

\begin{proof} The statements (1), (2) and (3) follow from definition. The proofs of (4), (5), (6), (7) and (8) can be found in \cite{Ba}. (9) has been proven in \cite{HM}.

\end{proof}

We give some examples of \qhol\ \fcts, which can be also found in \cite{Ba}.

\begin{ex} \label{ex-qhol}

\startenum

\new Every \ph\ or anti-holomorphic \fct\ is 1-holo\-morphic.

\new If $V$ is a complex submanifold of $U$, then every restriction of a \fct\ $f\in\ocal_q(U)$ to $V$ lies in $\ocal_q(V)$ because the inclusion mapping $i:V\hookrightarrow U$ is holomorphic.

\new \label{ex-qhol-locq} If there are local coordinates $z_1,\ldots,z_N$ such that a given $\ccal^2$-smooth \fct\ depends holomorphically in $N{-}q$ variables $z_1,\ldots,z_{N-q}$, then $g$ is \qhol.

\new \label{ex-qhol-char} Let $h=(h_1,\ldots,h_q)$ be a holomorphic mapping from $U$ into $\cbb^q$ and $V:=\{ z \in U : h(z)=0\}$. Then
$$
\chi_{m,V}=\frac{1}{1+m(h_1^2 + \ldots + h_q^2)}
$$
lies in $\ocal_q(U)$ for every $m\in\nbb$ due to the previous property (3). The sequence $(\chi_{m,V})_{m\in \nbb}$ decreases to the characteristic function $\chi_V$ of $V$ in $U$.

\new Let $L$ be an affine complex plane in $\cbb^N$ of codimension $q\in \{ 0,\ldots,N-1\}$ and let $U$ be an open set in $\cbb^N$. Consider a \fct\ $h\in \ccal^2(U)$ which is holomorphic on $U\cap L'$ for every parallel copy $L'$ of the plane $L$. Then by property (3) above the \fct\ $h$ lies in $\ocal_q(U)$.
\end{ex}

The last example leads to another subfamily of \qhol\ \fcts\ which will serve later for the characterization of the Shilov boundary of bounded convex domains.

\begin{defn}\label{defn-hol-linear} Let $U$ be an open set and $K$ be a compact set in $\cbb^N$. Let $L$ be an affine complex plane of codimension $q \in \{ 0,\ldots, N-1\}$.

\benum

\item We denote by $\ocal(L,U)$ the set of \fcts\ described in the last example, part (5).

\item The class $\ocal_q^\pi(U)$ is the union of all sets $\ocal(L,U)$, where $L$ varies among all affine complex planes in $\cbb^N$ of codimension $q$.

\item The set $\ocal(L,K)$ will mean the set of all \cont\ \fcts\ $f$ on $K$ such that there exist a \nbh\ $U$ of $K$ and a \fct\ $F \in \ocal(L,U)$ with $F|K=f$. The class $\ocal^\pi_q(K)$ is then the union of all sets $\ocal(L,K)$, where $L$ again varies among all affine complex planes of codimension $q$.

\item We set $A(L,K):=\ocal(L,\inti{K}) \cap \ccal(K)$ and $A^\pi_q(K):=\ocal^\pi_q(\inti{K}) \cap \ccal(K)$.
\eenum

\end{defn}

We have the following properties for this new class of \fcts.

\begin{prop} Let $U, K ,L$ and $q$ be as in the previous Definition \ref{defn-hol-linear}. Then we have the following properties. \startenum

\benum
\item $\ocal_q^\pi(U) \subset \ocal_q(U)$

\item $\ocal^\pi_0(U)=\ocal_0(U) \subset \ocal^\pi_1(U) \subset \ldots \subset \ocal^\pi_{N-1}(U) \subset \ccal^2(U)$

\item The family $A(L,K)$ and the uniform closure of $\ocal(L,K)$ are uniform subalgebras of $\ccal(K)$.
\eenum

\end{prop}

\bproof Property (1) follows from Example \ref{ex-qhol} (5). The  inclusions in point (2) follow directly from the definition. The statement in point (3) is easy to verify, since the uniform limit of a sequence of holomorphic \fcts\ remains holomorphic.
\eproof

\section{Shilov boundary for \qpsh\ \fcts}

We first prove the existence of the Shilov boundary for the subclasses of \qpsh\ \fcts\ defined in the previous section.

\begin{prop} \label{prop-qshilov} Let $U \subset\cbb^N$ be open, $K\subset \cbb^N$ be compact and $L$ be a complex plane in $\cbb^N$ of codimension $q \in \{ 0,\ldots,N-1\}$. Then we have the following properties.

\begin{enumerate}

\item Recall that $\log|\bcal|=\{ \log|f| : f \in \bcal\}$ for a subfamily $\bcal$ of complex-valued \cont\ \fcts. Then
    $$\log|\ocal_0(U)|\subset \log|\ocal(L,U)| \subset \log|\ocal_q^\pi(U)| \subset \log|\ocal_q(U)| \subset \PSH_q(U).$$

\item For the respecting Shilov boundaries we have that
$$
\check{S}_{\ocal_0(K)} \subset \check{S}_{\ocal_q(L,K)} \subset \check{S}_{\ocal_q^\pi(K)}\subset \check{S}_{\ocal_q(K)}\subset\check{S}_{\PSH_q(K)} \subset bK
$$
$$
\mathrm{and} \quad \check{S}_{A_0(K)} \subset \check{S}_{A_q(L,K)} \subset \check{S}_{A_q^\pi(K)}\subset \check{S}_{A_q(K)}\subset\check{S}_{\acal\PSH^0_q(K)} \subset bK.
$$

\item Let $\D \bcal \in \{ \ocal_0(K), \ \ocal(L,K), \ \ocal_q^\pi(K), \ \ocal_q(K), \ \PSH_q(K)\}$. Then
$$
\cl{P}_{\cl{\bcal}}=\check{S}_{\bcal} \in b_{\bcal} \qand \cl{P}_{\PSH_q(K)}=\check{S}_{\PSH_q(K)} \in b_{\PSH_q(K)}
$$

\item Let $\D \bcal \in \{ A_0(K), \ A(L,K), \ A_q^\pi(K), \ A_q(K), \ \acal\PSH^0_q(K)\}$. Then
$$
\cl{P}_{\bcal}=\check{S}_{\bcal} \in b_{\bcal}.
$$

\end{enumerate}

\end{prop}

\begin{proof} \startenum \new This follows from Example \ref{ex-qhol} (5) and Proposition \ref{prop-qhol} (9).

\new This is a consequence of part (1), Proposition \ref{prop-ashilov} (6), Remark \ref{rem-bshilov}, Proposition \ref{prop-qpsh} \eqref{qpsh-maxpr} and Corollary \ref{cor-shilov-cl} together with the following fact: If $f$ is a \fct\ from $A_q(K)$, then the \fcts\ $\psi_n:=\max\{ \log|f|,-n\}$, $n \in \nbb$, define a sequence of \fcts\ from $\acal\PSH_q(K)$ decreasing to $\log|f|$. Thus, $\log|A_q(K)|$ lies in $\scl{\acal\PSH_q(K)}$. Thus, $\cs_{A_q(K)}$ is contained in $\cs_{\acal\PSH_q(K)}$.

\new Let $\acal\in \{ \log|\bcal|,\PSH_q(K)\}$. It is obvious that $\acal_0:=\log|\ocal_0(K)|$ generates the topology of $K$. By Proposition \ref{prop-qhol} (5), (9) and Proposition \ref{prop-qpsh} (3) we deduce that the set $\acal$ is a scalar cone such that $\log|\ocal_0(K)| + \acal \subset \acal$. Then it follows from Theorem \ref{thm-shilov-bound} that $\SA$ is an $\acal$-boundary, so $\check{S}_{\acal} \in b_{\acal}$.

If $\bcal=\ocal_0(K)$ or $\bcal=\ocal(L,K)$, we can directly apply Bishop's Theorem to the Banach subalgebras $\cl{\bcal}$ in order to obtain $\cl{P}_{\cl{\bcal}}=\check{S}_{\cl{\bcal}}=\check{S}_{\bcal}$.

Let $\lbb_q$ be the set of all complex planes of codimension $q$ in $\cbb^N$. By Proposition \ref{prop-ashilov} (6), Bishop's Theorem and Proposition \ref{prop-pms} (4) we conclude that
$$
\check{S}_{\ocal_q^\pi(K)} = \cl{ \bigcup_{L' \in \lbb_q} \check{S}_{\ocal(L',K)}}=\cl{\bigcup_{L' \in \lbb_q} P_{\cl{\ocal(L',K)}}} \subset \cl{P}_{\bigcup_{L' \in \lbb_q}\cl{\ocal(L',K)}} \subset \cl{P}_{\cl{\ocal_q^\pi(K)}} \subset \check{S}_{\ocal_q^\pi(K)}.
$$
Given a \fct\ $f \in \ocal_q(K)$ let $\bcal_f$ be the uniform algebra in $\ccal(K)$ generated by $f$ and $\ocal_0(K)$. It follows from Proposition \ref{prop-qhol} (5) that this is really an algebra. We set $\mcal:=\bigcup_{f \in \ocal_q(K)} \bcal_f$. Then we have the inclusions $\ocal_q(K) \subset \mcal \subset \cl{\ocal_q(K)}$ and, thus, $\check{S}_\mcal=\check{S}_{\ocal_q(K)}=\check{S}_{\cl{\ocal_q(K)}}$. Now by Proposition \ref{bishop-unhol} we obtain that
$$
\check{S}_{\ocal_q(K)} = \check{S}_\mcal = \cl{P}_\mcal \subset \cl{P}_{\cl{\ocal_q(K)}} \subset \check{S}_{\cl{\ocal_q(K)}}=\check{S}_{\ocal_q(K)}.
$$
Hence, for $\bcal \in \{ \ocal_q^\pi(K), \ocal_q(K)\}$ we have that $\cl{P}_{\cl{\bcal}}=\check{S}_\bcal$.

Let $\acal=\PSH_q(K)$ and let $\Theta$ be the set of all non-negative $\ccal^2$-smooth \fcts\ on $\cbb^N$ with compact support. Then $\acal[\Theta]$ forms the so-called \textit{\sqpsh\ \fcts} \ on $K$. To see our final identity $\cl{P}_{\acal} = \check{S}_{\acal}$ we just have to apply Theorem \ref{bishop-sa} to the set $\acal$ and to the \fct\ $\omega(z):=1+|z|^2$, $z \in \cbb^N$.

\new By setting again $\acal_0:=\log|\ocal_0(K)|$ and by the same reasons as in the first part of the previous point (3), we can deduce that the Shilov boundary for the class $\bcal$ exists.

Since the families $A_0(K)$ and $A(L,K)$ are uniform subalgebras of $\ccal(K)$ we can apply Bishop's thereom in order to obtain the corresponding peak point property.

The family $A^\pi_q(K)$ is the union of uniform algebras of the form $A(L',K)$, where $L'$ varies among all complex planes $L'$ of codimension $q$. For given $f \in A_q(K)$ the family $A_f$ denotes the uniform closure of the algebra generated by $f$ and $A_0(K)$. The family $A_q(K)$ is then exactly the union of all such families $A_f$, where $f \in A_q(K)$. Then by the same arguments as in the middle part of the previous point, we obtain peak point properties for the families $A^\pi(K)$ and $A_q(K)$.

The last peak point property for the class $\acal\PSH^0_q(K)$ is again due to Theorem \ref{bishop-sa} by using $\omega(z):=1+|z|^2$, $z \in \cbb^N$.

\end{proof}

For certain subfamilies of \qpsh\ or \qhol\ \fcts\ we are already able to classify the Shilov boundary.

\brem \label{rem-trivial-qshilov} (1) Let $K$ be a compact set in $\cbb^N$. Then the Shilov boundary for the family
$$
\acal := \PSH_q(\inti K) \cap \USC(K)
$$
of all \qpsh\ \fcts\ on the interior of $K$ which are \usc\ up to the boundary $K$ is exactly the whole boundary of $K$. Indeed, by the maximum principle (Proposition \ref{prop-qpsh} \eqref{qpsh-maxpr}), the Shilov boundary for $\acal$ is contained in the boundary of $K$. On the other hand, pick a point $x$ in the boundary of $K$. Then the characteristic \fct\ $\chi_{\{x\}}$ of the set $\{x\}$ in $K$ lies in $\acal$. Moreover, it peaks at $x$. Hence, we have that the whole boundary of $K$ is the set $P_\acal$. Since this set lies in the Shilov boundary for $\acal$, we conclude that $\SA=bK$.

(2) The following \fct\ $f$ from Example 5 in \cite{Ba} is $(N-1)$-holomorphic on $\cbb^N\setminus\{0\}$ and has an isolated non-removable singularity at the origin,
$$
f(z)=\frac{\bar{z}_1+ \ldots + \bar{z}_N}{|z_1|^2+\ldots+|z_N|^2}.
$$
Let $p$ be a boundary point of a compact set $K$ in $\cbb^N$ and let $(p_n)_{n \in \nbb}$ be a sequence of points $p_n \notin K$ which converges to $p$ outside $K$. For $n \in \nbb$ consider the \fct\ $f_n(z):=f(z-p_n)$, which is $(N-1)$-holomorphic on $\cbb^N\setminus\{p_n\}$. Now if $n$ tends to $+\infty$, the absolute values $|f_n(p)|$ tend to $+\infty$. Hence, for every small enough \nbh\ $U$ of $p$ there is an index $n \in \nbb$ such that $U$ contains $p_n$ and $|f_n|$ attains its maximum on $K$ only inside the set $U\cap K$. By the definition of the Shilov boundary, the set $U$ intersects $\cs_{\ocal_{N-1}(K)}$. Since $U$ was an arbitrary small \nbh\ of $p$, the point $p$ itself is contained in $\cs_{\ocal_{N-1}(K)}$. Therefore, the whole boundary $bK$ of $K$ is contained $\cs_{\ocal_{N-1}(K)}$. Now the local maximum modulus principle in Proposition \ref{prop-qhol} \eqref{loc-max-qhol} yields
$$
\cs_{\ocal_{N-1}(K)} = bK.
$$

\erem

In the next statement we compare the Shilov boundary for subclasses of \qhol\ \fcts\ defined on subspaces of different dimensions. Some ideas of its proof are similar to the arguments given in the proof of Theorem 3 in \cite{Ba2}.

\begin{prop} \label{prop-lower-shilov}

 Let $K$ be a compact set in $\cbb^N$ which admits a Stein \nbh\ basis. Given a complex plane $L$ of codimension $q\in\{ 0,\ldots,N-1\}$ it holds that
$$
\check{S}_{\ocal(K \cap L)} = \check{S}_{\ocal(L,K)} \cap L.
$$
\end{prop}

\begin{proof}

Observe that $K \cap L$ is non-empty if and only if $\check{S}_{\ocal(L,K)}\cap L$ is non-empty. Indeed, assume that $K\cap L$ is non-empty, but $L$ does not intersect $S_0:=\check{S}_{\ocal(L,K)}$. For $n \in \nbb$ let $\chi_{n}:=\chi_{n,L}$ be the \fcts\ from the part (4) of Example \ref{ex-qhol}. It is obvious that $\chi_{n} \in \ocal(L,K)$, since it is constant on each plane of codimension $q$ parallel to $L$. Recall that $\chi_{n}$ decreases to the characteristic \fct\ of $L$. Then for large enough integer $n \in \nbb$, we can arrange that
$$
\max_{K} \chi_{n} \geq \max_{K \cap L} \chi_{n} > \max_{S_0} \chi_{n},
$$
which is a contradiction to the definition of the Shilov boundary for the class $\ocal(L,K)$. The other direction is obvious, because $\check{S}_{\ocal(L,K)}$ is a non-empty subset of $K$ due to Proposition \ref{prop-qshilov} (3).

We continue by proving the inclusion $\check{S}_{\ocal(K \cap L)} \subset \check{S}_{\ocal(L,K)} \cap L$. Let again be $S_0:=\check{S}_{\ocal(L,K)}$. We have to show that $\max_{K \cap L}|f| = \max_{S_0 \cap L}|f|$ for every \fct\ $f \in \ocal(K \cap L)$. Pick an arbitrary \fct\ $f \in \ocal(K\cap L)$. Then $f \in \ocal(U\cap L)$ for some open \nbh\ $U$ of $K$. Since $K$ has a Stein \nbh\ basis, we can assume that $U$ is \psc. Let $F$ be a holomorphic extension of $f$ to the whole of $U$. Then $F_{n}:= F \cdot\chi_{n,L} \in \ocal(L,K)$ for every $n \in \nbb$. Furthermore, we have that
$$
\max_{K \cap L}|f| = \lim_{n \to \infty} \max_{K}|F_{n}| = \lim_{n \to \infty} \max_{S_0}|F_{n}| = \max_{S_0 \cap L}|f|.
$$
By the definition it means that $\check{S}_{\ocal(K \cap L)}$ is contained in $S_0 \cap L$. For the other inclusion, take a point $p \in L \cap P_{\cl{\ocal(L,K)}}$. Then there is a peak \fct\ $f$ in $\cl{\ocal(L,K)}$ such that $\{ z \in K : |f(z)|=\max_K|f|\}=\{ p\}$. It is easy to see that $g=f|L$ lies in $\cl{\ocal(K\cap L)}$ and that $\{ z \in K \cap L : |g(z)|=\max_{K \cap L}|g|\}=\{ p\}$, because $p \in K \cap L$. Thus, $p$ is also a peak point for $\cl{\ocal(K\cap L)}$. We obtain that
$$
L \cap P_{\cl{\ocal(L,K)}} \subset P_{\cl{\ocal(K\cap L)}} \subset \check{S}_{\cl{\ocal(K\cap L)}}=\check{S}_{\ocal(K\cap L)}.
$$
Together with Proposition \ref{prop-qshilov} (3) we conclude that
$$
L \cap \check{S}_{\ocal(L,K)} = L \cap \cl{P}_{\cl{\ocal(L,K)}} \subset \check{S}_{\ocal(K\cap L)}.
$$
\end{proof}

\section{Generalization of Bychkov's theorem}

In \cite{By}, S.N.Bychkov gave a characterization of the Shilov boundary for bounded convex domains $D \subset \cbb^N$. Our goal in this section is to generalize this theorem to Shilov boundaries for subclasses of \qpsh\ and \qhol\ \fcts\ (see Theorem \ref{Cqbyckov} below).

First, we introduce one more subclass of \cont\ \fcts\ which is usually used when working with the classical Shilov boundary for holomorphic \fcts. Namely, given a compact set $K$ in $\cbb^N$, the set $A_0(K):=\ocal_0(\inti(K))\cap \ccal(K)$ forms a uniform subalgebra of $\ccal(K)$.

We recall the main result of Bychkov's article \cite{By}.

\begin{thm}[Bychkov, 1981] Let $D$ be a bounded convex open set in $\cbb^2$. A boundary point $p \in bD$ is not in $\check{S}_{A_0(\cld)}$ if and only if there is a \nbh \ $U$ of $p$ in $bD$ such that $U$ consists only of complex points (see Definition \ref{Cdefqcomplex}).
\end{thm}

\brem\label{rem-snhbcvx} If $D \relc \cbb^N$ is a bounded convex domain, it is easy to verify that $D$ has a Stein \nbh\ basis and that $$
\check{S}_{A^\pi_q(\cld)}=\check{S}_{\ocal^\pi_q(\cld)}, \quad \check{S}_{A_q(\cld)}=\check{S}_{\ocal_q(\cld)} \qand \check{S}_{\acal\PSH_q^0(\cld)}=
\check{S}_{\PSH_0(\cld)\cap\ccal{\cld}}.
$$

\erem

We recall some definitions from convexity theory given in Bychkov's aricle \cite{By}. We also mainly use his notations.

\begin{defn}\label{def-cvx} \

\benum

\item A set $K\subset\rbb^m$ is called \textit{convex} if for every two points $x_1$, $x_2$ contained in $K$ the segment $[x_1,x_2]=\{ (1-t)x_1+tx_2 : 0 \leq t \leq 1\}$ also lies in $K$. The dimension of the smallest (real) plane containing $K$ is the \textit{dimension of $K$}.

\item Let $K$ be a \textit{convex body}, i.e., a compact convex set with non-empty interior, and let $p$ a boundary point of $K$. Every hyperplane $H$ in $\rbb^m$ splits the space $\rbb^m$ into two halfspaces $H^+$ and $H^-$. The hyperplane $H$ is then said to be \textit{supporting for $K$ at $p$} if $H$ contains $p$ and $K$ lies in one of the closed halfspaces $H \cup H^+$ or $H \cup H^-$.

\item A subset of the boundary $bK$ of $K$ which results from an intersection of $K$ with supporting hyperplanes is called a \textit{face of $K$}. A face is again a lower dimensional convex set. The empty set and $K$ itself are also considered to be faces. A face of a face of $K$ does not need to be a face of $K$. The arbitrary intersection of faces of $K$ is again a face of $K$.

\eenum

\end{defn}

\brem\label{rem-minface}

Given a convex body $K$, there exists a unique minimal face $F_1=F_{\mathrm{min}}(p,K)$ of $F_0:=K$ in the boundary of $K$ containing the point $p$. It can be defined as the intersection of $K$ and all supporting hyperplanes for $K$ at $p$. Then there are two options for $p$: either it is an inner point of the convex body $F_1$ or it lies on the boundary of $F_1$. In the second case, the point $p$ might again lie either in the interior of the minimal face $F_2=F_{\mathrm{min}}(p,F_1)$ of $F_1$ or in the boundary of $F_2$. Inductively, we obtain a finite sequence $(F_j)_{j=0,\ldots, j(p)}$ of convex bodies $F_j$ in $K$ of dimension $m_j$ such that $F_{j+1}=F_{\mathrm{min}}(p,F_j) \supset F_j$ for each $j\in\{0,\ldots, j(p)-1\}$ and such that either, if $m_{j(p)}>0$, the point $p$ is an interior point of $F_{j(p)}$, or, if $m_{j(p)}=0$, the minimal face $F_{j(p)}$ consists only of the point $p$.

\erem

\bdefn The convex body $F_p(K):=F_{j(p)}$ obtained in Remark \ref{rem-minface} above will be called the \textit{face essentially containing} $p$. It is contained in a plane $E_p(K)$ of dimension $m_{j(p)}$ which admits $E_{j(p)} \cap K=F_{j(p)}$.
\edefn

\begin{ex} Let $\Delta$ be the unit disc in $\rbb^2$ and consider the set $K=\cl{\Delta} \cup ([-1,1]\times[-1,0])$. It is a convex body in $\rbb^2$. The plane $\pi_1=\{ 1\}\times \rbb$ is the only supporting hyperplane of $K$ at $p=1$ in $\rbb^2$. Thus, the minimal face of $K$ containing $p$ is the segment $F_1=\pi_1\cap K = \{ 1\}\times[-1,0]$. The point $p$ lies in the boundary of $F_1$ in $\{ 1\}\times \rbb$. Then the set $\pi_2:=\{ 1 \}$ is the only supporting hyperplane of $F_1$ at $p$ in $\{ 1\}\times \rbb$. Hence, the minimal face of $F_1$ having $p$ inside is the set $F_2=\pi_2 \cap K = \{ p\}$. Therefore, the face essentially containing $p$ is the set $F_2=\{ 1\}$.
\end{ex}

In the following, let $D$ be always a bounded convex domain in $\cbb^N$.

\begin{defn}\label{Cdefqcomplex}
Let $p\in bD$ and let $E_p^{\cbb}(\cld)$ be the largest complex plane inside $E_p(\cld)$ passing through $p$. We define $\nu(p)$ to be the complex dimension of $E_p^{\cbb}(\cld)$. If $\nu(p)=0$, then $E_p(\cld)$ is totally real and we say that the point $p$ is \textit{real}.

The set $\Pi_p(\cld)$ will denote the set of all complex planes $\pi$ in $\cbb^N$ such that there exists a domain $G\subset\cbb^N$ with $p\in G\cap \pi\subset bD$. If $\Pi_p(\cld)$ is not empty, then $p$ is called \textit{complex}.
\end{defn}

We restate Lemma 2.5 in \cite{By} and its important corollary.

\begin{lem}\label{Clembyc}  If $I\subset bD$ is an open segment containing $p\in bD$, then $I\subset F_p(\cld)$.
\end{lem}

\begin{cor}\label{Ccorbyc} A boundary point $p\in bD$ is either real or complex.
\end{cor}

From this we can derive further consequences.

\begin{cor}\label{lem-max-c-plane} If $p \in bD$ is complex, then $E_p^{\cbb}(\cld) \in \Pi_p(\cld)$.
\end{cor}

\begin{proof} Since $p$ is complex, it can not be real due to the previous Corollary \ref{Ccorbyc}. Thus, $E_p^{\cbb}(\cld)$ is not empty and the face essentially containing $p$ can not be a single point. The point $p$ is then an inner point of the convex body $F_p(\cld)$ in $E_p(\cld)$. Hence, there is an open ball $B'$ with center $p$ inside $F_p(\cld)$, and we can find an open ball $B$ in $\cbb^N$ with center $p$ such that $B \cap E_p(\cld)=B'$. Then we obtain that
$$
E_p^{\cbb}(\cld) \cap B \subset E_p(\cld) \cap B=B' \subset F_p(\cld) \subset bD.
$$
It follows now from the definition of $\Pi_p(\cld)$ that $E_p^{\cbb}(\cld)$ lies in $\Pi_p(\cld)$.
\end{proof}

\begin{cor}\label{cor-pi-unique} If $\pi \in \Pi_p(\cld)$, then $\pi$ lies in $E_p^{\cbb}(\cld)$.
\end{cor}

\begin{proof} Let $G$ be an open \nbh\ of $p$ in $\cbb^N$ such that $p\in G \cap \pi \subset bD$. It follows from Lemma \ref{Clembyc} that $p\in G \cap \pi$ lies in $F_p(\cld)$. Since $G \cap \pi$ is open in $\pi$, we have that $\pi$ is contained in $E_p(\cld)$. Since $\pi$ is a complex plane containing $p$ and $E_p^{\cbb}(\cld)$ is the largest complex plane inside $E_p(\cld)$, we conclude that $\pi$ lies in $E_p^{\cbb}(\cld)$.
\end{proof}

We specify complex points in the following way.

\bdefn Given $q\in\{ 1,\ldots,N-1\}$, a complex point $p$ is called \textit{$q$-complex} if $\nu(p)\geq q$.
\edefn

The previous classifications can be reduced to the following simple observation.

\brem A boundary point $p$ in $bD$ is $q$-complex if and only if there is a domain $G$ in $\cbb^N$ and a complex plane of dimension at least $q$ such that $p \in G \cap \pi \subset bD$.
\erem

The next lemma asserts that a complex point $p$ is a lower dimensional real point when intersecting the convex body with a complex plane containing $p$ transversal to $E_p^{\cbb}(\cld)$.

\begin{lem}\label{Clembycreal} Let $p$ be a complex point in $bD$. Let $\pi$ be a complex affine plane of codimension $\nu(p)$ such that $E_p^{\cbb}(\cld)\cap\pi=\{ p\}$. Then $p$ is a real boundary point of $\cld \cap \pi$.
\end{lem}

\begin{proof} If $\nu(p)=N{-}1$, the statement is obviously true, since every boundary point of $\cld \cap \pi$ is real.

Suppose that $\nu(p) \leq N{-}2$ and that the statement is false. Then by Corollary \ref{Ccorbyc} the point $p$ is a complex boundary point of $\cld \cap \pi$. By Corollary \ref{lem-max-c-plane} there exist a domain $G\subset\cbb^N$ and a complex line $\lbb$ in $\pi$ such that $p \in G\cap \lbb \subset bD\cap \pi \subset bD$. Hence, $\lbb \in \Pi_p(\cld)$. By Corollary \ref{cor-pi-unique} the line $\lbb$ lies in $E_p^{\cbb}(\cld)$. But since $E_p^{\cbb}(\cld)\cap\pi=\{ p\}$ and $\lbb \subset \pi$, it follows that $\lbb=\{ p\}$, which is absurd.
\end{proof}

We generalize now Proposition 2.6 in \cite{By} which states that a real boundary point always lies in the Shilov boundary for the class $A_0(\cld)$.

\begin{prop} \label{Cqcomplexsilov} If $p \in bD$, then $p \in \check{S}_{\ocal_{\nu(p)}^\pi(\cld)}$.
\end{prop}

\begin{proof} By Corollary \ref{Ccorbyc}, $p$ is either real or complex.

If $p$ is real, then by Proposition 2.6 in \cite{By} and Remark \ref{rem-snhbcvx} we have that
$$
p \in \check{S}_{A_0(\cld)}=\check{S}_{\ocal_0(\cld)}=\check{S}_{\ocal_0^\pi(\cld)}.
$$
Recall that it follows from definition that $\ocal_0(\cld)=\ocal_0^\pi(\cld)$.

If $p$ is complex, then there are a domain $G$ and a $\nu(p)$-dimensional complex plane $\pi$ such that $p \in G \cap \pi \subset bD$. Let $L$ be a complex affine plane of codimension $\nu(p)$ such that $\pi\cap L = \{ p\}$. Then, by Lemma \ref{Clembycreal}, the point $p$ is a real boundary point of the convex body $\cld \cap L$. By Proposition 2.6 in \cite{By} and by Propositions \ref{prop-lower-shilov} and \ref{prop-qshilov} (2) we obtain that
$$
p \in \check{S}_{A_0(\cl{D}\cap L)}=\check{S}_{\ocal_0(\cl{D}\cap L)} \subset \check{S}_{\ocal(L,\cld)} \subset \check{S}_{\ocal_{\nu(p)}^\pi(\cld)}.
$$
\end{proof}

As a first consequence, we obtain a characterization of the Shilov boundary for the family of $(N-1)$-\psh\ \fcts. Compare also Remark \ref{rem-trivial-qshilov} (2).

\begin{cor}\label{CN-1byckov} The Shilov boundaries for the classes $\ocal^\pi_{N-1}(\cld)$, $\ocal_{N-1}(\cld)$ and $\PSH_{N-1}(\cld)$ coincide with the topological boundary of $D$; i.e.,
$$
\check{S}_{\ocal^\pi_{N-1}(\cld)}=\check{S}_{\ocal_{N-1}(\cld)}=\check{S}_{\PSH_{N-1}(\cld)}=bD.
$$
\end{cor}

\begin{proof} If $p \in bD$, then $p$ is real or complex and $0\leq \nu(p)\leq N-1$. Thus, the previous proposition and Propositions \ref{prop-qshilov} (2) and \ref{prop-qpsh} (10) imply that
$$
p \in \check{S}_{\ocal^\pi_{\nu(p)}(\cld)}\subset \check{S}_{\ocal^\pi_{N-1}(\cld)}\subset \check{S}_{\ocal_{N-1}(\cld)} \subset \check{S}_{\PSH_{N-1}(\cld)} \subset bD.
$$
\end{proof}

We will need the following lemma.

\begin{lem}\label{lem-complex-peak} Let $p \in bD$ and $q \in \{ 0,\ldots,N-2\}$. If there exists an at least ${(q+1)}$-dimensional complex analytic set in $bD$ containing $p$, then $p$ is not a peak point for the class $P_{\PSH_q(\cld)}$. In particular, no $(q+1)$-complex point can be contained in $P_{\PSH_q(\cld)}$.
\end{lem}

\begin{proof} This follows immediately from the local maximum principle for \qpsh\ \fcts\ on analytic sets (see Corollary 5.3 in \cite{Sl}).
\end{proof}

We are now able to generalize Bychkov's theorem.

\begin{defn} For $q \in \{ 1,\ldots,N{-}1\}$ denote by $\Gamma_{q}(\cld)$ the set of all boundary points of $D$ which have a \nbh\ $U$ in $bD$ such that $U$ consists only of $q$-complex points.
\end{defn}

\begin{thm}\label{Cqbyckov} Let $q \in \{ 0,\ldots,N-2\}$. Then
$$
\check{S}_{\ocal^\pi_q(\cld)}=\check{S}_{\ocal_q(\cld)}=\check{S}_{\PSH_q(\cld)}=bD\setminus \Gamma_{q+1}(\cld).
$$
\end{thm}

\begin{proof} If $p \in bD\setminus\check{S}_{\ocal^\pi_q(\cld)}$, then there is a \nbh \ $U$ of $p$ in $bD$ such that $U \cap \check{S}_{\ocal^\pi_q(\cld)} =\emptyset$. Thus, if $w \in U$, then $\nu(w)\geq q+1$ due to Proposition \ref{Cqcomplexsilov}. This means that $U$ consists only of $(q+1)$-complex points. Hence, $p\in\Gamma_{q+1}(\cld)$. We conclude that
$$
bD\setminus\Gamma_{q+1}(\cld)\subset \check{S}_{\ocal^\pi_q(\cld)}.
$$
On the other hand, if there is a \nbh \ $U$ of $p$ in $bD$ such that $U$ contains only $(q+1)$-complex points, then, by Lemma \ref{lem-complex-peak}, we obtain that $U \cap P_{\PSH_q(\cld)} = \emptyset$. This implies that $p \notin \cl{P}_{\PSH_q(\cld)}$. Since, by Proposition \ref{prop-qshilov} (3), the latter set coincides with $\check{S}_{\PSH_q(\cld)}$, we obtain the other inclusion
$$
\check{S}_{\PSH_q(\cld)} \subset bD\setminus\Gamma_{q+1}(\cld).
$$

In view of Proposition \ref{prop-qshilov} (2) this completes the proof.
\end{proof}

Now we give an interesting observation following from the previous Theorem.

\begin{rem}\label{rem-qshilov-analyt} Given an integer $q\in\{ 1,\ldots,N-1\}$ let $\Gamma^A_q(\cld)$ be the set of all boundary points $p$ of $D$ such that there exists a \nbh\ $U$ of $p$ in $bD$ so that for each point $z \in U$ there is a complex analytic set in $U$ of dimension at least $q$ containing $z$. Then
$$
\Gamma^A_q(\cld)=\Gamma_q(\cld).
$$
Indeed, the inclusion $\Gamma_q(\cld) \subset \Gamma^A_q(\cld)$ follows directly from the definition of these two sets and the definition of $q$-complex points.

Now let $p \in \Gamma^A_q(\cld)$. Then Lemma \ref{lem-complex-peak} and Proposition \ref{prop-qshilov} (3) imply that $p \notin \check{S}_{\PSH_{q-1}(\cld)}$. Thus, by Theorem \ref{Cqbyckov} we have that $p$ is contained in $\Gamma_q(\cld)$. This shows the other inclusion.
\end{rem}

In the end of this section, we check for an analytic structure of the Shilov boundary of \qpsh\ \fcts\ on convex sets.

\bthm\label{thm-fol-cvx} Let $q\in \{ 1,\ldots,N-1\}$ and assume that $\{z\in bD : \nu(z)\geq q+1\}$ is open. If it is non-empty, then the following open part
$$
\fcal_q(\cld):=\inti_{bD}\left(\check{S}_{\PSH_q(\cld)}\setminus\check{S}_{\PSH_{q-1}(\cld)}\right)
$$
of the Shilov boundary for $\PSH_q(\cld)$ in $bD$ locally admits a complex foliation by complex $q$-dimensional planes in the following sense: for every point $p \in \fcal_q(\cld)$ there exists a \nbh\ $U$ of $p$ in $bD$ such that for each $z \in U$ there is a domain $G_z$ in $\cbb^N$ and a unique complex $q$-dimensional plane $\pi_z$ with $z \in \pi_z \cap G_z \subset U$. In the special case $q=N-1$, these complex (hyper-)planes are aligned parallelly.
\ethm

\bproof We set $\Gamma_N:=\emptyset$. By Theorem \ref{Cqbyckov} and by Corollary \ref{CN-1byckov} we have that $\fcal_q(\cld)=\Gamma_{q}(\cld)\setminus\cl{\Gamma_{q+1}(\cld)}$. If the set $\{z\in bD : \nu(z)\geq q+1\}$ is open, then it coincides with $\Gamma_{q+1}(\cld)$. Thus,
$$
\fcal_q(\cld)=\Gamma_{q}(\cld)\setminus\cl{\{z\in bD : \nu(z)\geq q+1\}}.
$$
Now if $p$ is an arbitrary point from $\fcal_q(\cld)$, then there is a \nbh\ $W$ of $p$ in $\fcal_q(\cld)$ such that $\nu(z)=q$ for every point $z \in W$. Hence, the open set $\fcal_q(\cld)$ consists only of \textit{exactly $q$-complex points}. Then Corollaries \ref{lem-max-c-plane} and \ref{cor-pi-unique} imply existence and uniqueness of an open part of a complex $q$-dimensional plane $\pi_z=E^\cbb_z(\cld)$ containing $z$ and lying in $U$.

For the special case of $q=N-1$ the set $\fcal_{N-1}(\cld)$ is a convex hypersurface foliated by complex hyperplanes. By a result of Beloshapka and Bychkov in \cite{BeBy}, they have to be aligned parallelly. (See also Example \ref{ex-non-paral} and the remark before this example.)
\eproof

At the end of this section, we give an example for a convex domain $D$ in $\cbb^3$ such that the part $\fcal_1(\cld)$ does not admit a foliation in the sense of the previous theorem if the assumption on the openess of $\{z\in bD : \nu(z)\geq 2\}$ is dropped.

\bex Consider the domain $G$ in $\cbb\times \rbb$ given by $$G=\{ (x,y,u) \in \cbb\times\rbb : x^2 + (1-y^2)u^2 < (1-y^2), \ |y| < 1\}.$$
It is easy to compute that the \fct\ $h(y,u):=\sqrt{(1-y^2)(1-u^2)}$ is concave on $[-1,1]^2$. Since $G$ is the intersection of the sublevel set $\{x<h(y,u)\}$ of the concave \fct\ $h$ and the superlevel set $\{x>-h(y,u)\}$ of the convex \fct\ $-h$ over $[-1,1]^2$, it is convex in $\cbb\times\rbb$.

The boundary of $G$ contains the \textit{flat} parts $\{\pm i\}\times(-1,1)$ and $\{0\}\times[-1,1]\times\{\pm 1\}$ whereas the rest of the boundary consists of \textit{strictly convex} points. By puttting $D:=G\times(-1,1)^3$ we obtain a convex domain $D$ in $\cbb^3$ such that
$$
\{z\in bD : \nu(z)\geq 2\} = \{\pm i\}\times(-1,1)^4.
$$
and $\Gamma_1(\cld)$ is the whole boundary of $D$. In particular, $\Gamma_2(\cld)$ is empty. Thus, the boundary points $z$ in $bD$ with $\nu(z)\geq 2$ lie in $\Gamma_1(\cld)$, but there is no unique foliation by complex one-dimensional planes near these points.
\eex 

\section{Hausdorff dimension of the Shilov boundary}

In this section we prove some estimates on the Hausdorff dimension of the Shilov boundary for \qpsh\ \fcts\ on convex bodies.

\begin{defn} Let $(X,d)$ be a metric space. \startenum

\new For a subset $U$ of $X$ denote by $\diam(U)$ the \textit{diameter of $U$}, i.e.,
$$
\diam(U):=\sup\{ d(x,y): x,y \in U\}.
$$

\new Given a subset $E$ of $X$ and positive numbers $s$ and $\eps$ we set
$$
H^s_\eps(E):=\inf \left\{ \sum_{i=1}^\infty \diam(U_i)^s : E \subset \bigcup_{i=1}^\infty U_i, \ \diam(U_i)<\eps \ \forall \ i \in \nbb\right\}.
$$
The \textit{$s$-dimensional Hausdorff measure} is then defined by
$$
H^s(E):=\lim_{\eps \nach 0} H^s_\eps(E).
$$

\new For every subset $E$ of $X$ there is a number $s_0 \in [0,+\infty]$ such that
$$
H^s(E) = \infty \ \mathrm{for} \ s \in (0,s_0) \qand H^s(E)=0 \ \mathrm{for} \ s \in (s_0,\infty).
$$
The number $\dim_H E:=s_0$ is called the \textit{\textit{Hausdorff} (or metric) dimension of $E$}.
\end{defn}

The next statement can be found in, e.g., \cite{Fa}, Corollary 7.12.

\begin{prop}\label{prop-hausd} Let $I$ be a $m$-dimen\-sional cube in $\rbb^m$, $J$ be a $n$-dimen\-sional cube in $\rbb^n$ and $F$ be a subset of $I\times J$. For a given point $x \in I$ consider the slice $F_x:=F \cap (\{ x\}\times J)$. If $\dim_H F_x \geq \alpha$ for every $x \in I$, then $\dim_H F \geq \alpha+m$.
\end{prop}

It was shown in \cite{By} that the Hausdorff dimension of the Shilov boundary of a convex body in $\cbb^2$ is not less than 2. We partially generalize this result.

\begin{thm} Let $D$ be a convex bounded domain in $\cbb^N$ and $q\in\{ 0,\ldots,N-2\}$. Suppose that there are a constant $\alpha \geq 0$ and a complex $q$-codimensional plane $\pi_0$ intersecting $D$ such that
$$
\dim_H \check{S}_{\ocal_0(\cld \cap \pi)} \geq \alpha
$$
for every complex $q$-codimensional plane $\pi$ which lies nearby $\pi_0$ and which is parallel to $\pi_0$. Then
$$
\dim_H \check{S}_{\ocal_{q}(\cld)} \geq \alpha + 2q.
$$
In particular, $\dim_H \check{S}_{\ocal_{N-2}(\cld)} \geq 2N-2$.
\end{thm}

\begin{proof} Denote by $\Pi$ the set of the complex planes mentioned in the assumptions of this theorem. Then, by  Proposition \ref{prop-lower-shilov} and Proposition \ref{prop-qshilov}, we have that
$$
\bigcup_{\pi \in \Pi} \check{S}_{\ocal(\cld \cap \pi)} \subset \bigcup_{\pi \in \Pi}\check{S}_{\ocal(\pi, \cld)} \subset \check{S}_{\ocal^\pi_q(\cld)}\subset \check{S}_{\ocal_q(\cld)}.
$$
It follows then from Proposition \ref{prop-hausd} that $\dim_H\check{S}_{\ocal_{q}(\cld)} \geq \alpha + 2q$.

Let now $q=N-2$. It was shown in \cite{By}, Theorem 3.1, that $\dim_H \check{S}_{\ocal_0(\cld \cap \pi)} \geq \alpha=2$ for every complex two dimensional affine plane $\pi$ such that $\pi \cap D\neq \emptyset$. Hence, we conclude that
$$
\dim_H \check{S}_{\ocal_{N-2}(\cld)} \geq 2 + 2(N-2)=2N-2.
$$
\end{proof}

To show that the Hausdorff dimension of $\check{S}_{A_0(\cld)}$ is not less than two if $D \relc \cbb^2$ is a convex domain, Bychkov used that $\Gamma_1(\cld)$ admits a local foliation by complex lines which are aligned parallelly to each other. More general, if a convex hypersurface (i.e., an open part of the boundary of a convex body) is foliated by complex hyperplanes, then, by a result of Beloshapka and Bychkov in \cite{BeBy}, they are always parallel to each other. Especially, this holds for the open set $\Gamma_{N-1}(\cld)$, provided it is not empty. But, in general, it fails for lower dimensional complex foliations as the following example from \cite{NiTh} shows.

\begin{ex}\label{ex-non-paral} Consider the \fct\ $\vrho(z)=(\Re z_2)^2-(\Re z_1 )(\Re z_3 )$ for $z \in \cbb^3$. Then the set
$$
D:=\{z \in \cbb^3: \Re(z_1)>0, \vrho(z)<0 \}
$$
is convex and an open part of its boundary is foliated by a real 3-dimensional parameter family of open parts of non-parallel complex lines of the form
$$
\{ (a^2\zeta+ib, a\zeta +ic,\zeta), \zeta \in \cbb\}, \quad a,b,c \in \rbb.
$$

\end{ex} 

\section{Shilov boundary for smooth \qpsh\ \fcts}

In this section, we give a characterization of the Shilov boundary for $\ccal^2$-smooth \qpsh\ \fcts\ defined near the closure of a compact set.

\bdefn \label{def-more-qpsh}Let $K$ be a compact set in $\cbb^N$. \benum

\item We denote by $\PSH_q^2(K)$ the set of all \fcts\ which are $\ccal^2$-smooth and \qpsh\ in some \nbh\ of $K$.

\item The set $\PSH_q^c(K)$ is the set of all \fcts\ which are \cont\ on some \nbh\ of $K$ and locally the maximum of finitely many $\ccal^2$-smooth \qpsh\ \fcts.

\item The set $\PSH_q^0(K)$ is formed by all \fcts\ which are \cont\ and \qpsh\ in some \nbh\ of $K$.
\eenum
\edefn

\brem \label{rem-qshilov-2} Since $\acal_0:=\PSH^2_0(K)$ generates the topology of $K$ and fulfills $\acal_0+\acal \subset \acal$, where $\acal \in \{ \PSH_q^0(K), \PSH_q^c(K), \PSH_q^2(K)\}$, we can apply Theorem \ref{thm-shilov-bound} in order to obtain that $\cs_{\acal}$ is a non-empty $\acal$-boundary. If we put $\omega(z):=1+|z|^2$, $z \in \cbb^N$, then by Theorem \ref{bishop-sa} we get the peak property $\cl{P}_{\acal}=\cs_{\acal}$ for the subfamilies of \qpsh\ \fcts\ defined above.
\erem

In the following, we present a useful regularization technique derived from \cite{Dem}, Lemma (5.18) in {chapter 5}.

\begin{defn}\label{uscreg}

Let $\theta$ be a non-negative $\ccal^\infty$-smooth \fct\ on $\rbb$ with compact support in the unit interval $(-1,1)$ such that $\int_\rbb \theta(s)ds=1$ and $\theta(-t)=\theta(t)$ for all $t\in\rbb$. Given positive numbers $\eps_1,\ldots,\eps_l \in (0,+\infty)$ and $t_1,\ldots,t_l \in \rbb$, we define the \textit{regularized maximum} by
$$
\rmax{(\eps_1,\ldots,\eps_l)}( t_1,\ldots,t_l):=\int_{\rbb^l} \max\{ t_1+\eps_1 s_1,\ldots,t_l+\eps_l s_l\} \theta(s_1)\cdot\ldots\cdot\theta(s_l)d(s_1,\ldots,s_l).
$$
For a single positive number $\eps>0$ we set $\rmax{\eps}:=\rmax{(\eps,\ldots,\eps)}$.
\end{defn}

The regularized maximum has the following properties.

\begin{lem} \label{lem-reg-max}\

\benum

\item The \fct\ $( t_1,\ldots,t_l) \mapsto \rmax{(\eps_1,\ldots,\eps_l)}( t_1,\ldots,t_l)$ is a $\ccal^\infty$-smooth convex \fct\ on $\rbb^l$ which is non-decreasing in every variable $t_1,\ldots,t_l$.

\item It holds that $\max\{t_1,\ldots,t_l\} \leq \rmax{(\eps_1,\ldots,\eps_l)}( t_1,\ldots,t_l) \leq \max\{t_1+\eps_1,\ldots,t_l+\eps_l\}$.

\item If $t_j+\eps_j < \max_{i \neq j}\{ t_i-\eps_i\}$, then we have that
$$
\rmax{(\eps_1,\ldots,\eps_l)}(t_1,\ldots,t_l) = \rmax{(\eps_1,\ldots,\eps_{j-1},\eps_{j+1},\ldots,\eps_l)}(t_1,\ldots,t_{j-1},t_{j+1},\ldots,t_l).
$$
\eenum
\end{lem}

We can apply the regularized maximum to \qpsh\ \fcts.

\blem\label{lem-reg-max-qpsh} Let $\psi_1,\ldots,\psi_k$ be finitely many $\ccal^2$-smooth \fcts\ on an open set $U$ in $\cbb^N$ such that for each $j \in \{ 1,\ldots,k\}$ the \fct\ $\psi_j$ is $q_j$-\psh\ on $U$. Then for every tuple of positive numbers $(\eps_1,\ldots,\eps_k)$ the regularized maximum $\rmax{(\eps_1,\ldots,\eps_k)}\{ \psi_1,\ldots,\psi_k\}$  is $\ccal^2$-smooth and \qpsh\ on $U$, where $q=q_1+\ldots+q_k$.
\elem

\bproof This is a consequence of Lemma \ref{lem-reg-max} and Proposition 2.11 in \cite{TPESZ}.

\eproof

The regularized maximum allows to compare the Shilov boundaries of the families of smooth and non-smooth \qpsh\ \fcts\ introduced in Definition \ref{def-more-qpsh}.

\bprop\label{prop-more-qpsh-shilov} Given a compact set $K$ in $\cbb^N$ we have that
$$
P_{\PSH_q^2(K)} = P_{\PSH_q^c(K)}
$$
and
$$
\cs_{\PSH_q^2(K)} = \cs_{\PSH_q^c(K)}=\cs_{\PSH_q^0(K)}=\cs_{\PSH_q(K)}.
$$
\eprop

\bproof Since $\PSH_q^2(K) \subset \PSH_q^c(K)\subset \PSH_q^0(K)\subset \PSH_q(K)$, we derive for the set of peak points of these classes that
\bea
P_{\PSH_q^2(K)} \subset P_{\PSH_q^c(K)}\subset P_{\PSH_q^0(K)}\subset P_{\PSH_q(K)} \label{eqqq}
\eea
By the peak point property $\cl{P}_{\acal}=\SA$ for these families (see Remark \ref{rem-qshilov-2}) it follows that
$$
\cs_{\PSH_q^2(K)} \subset \cs_{\PSH_q^c(K)}\subset \cs_{\PSH_q^0(K)}\subset \cs_{\PSH_q(K)}.
$$

Assume now that there is a \fct\ $\psi \in \PSH_q^c(K)$ such that $\psi$ peaks at some point $p \in bK$. Then there are a \nbh\ $U$ of $p$ and finitely many $\ccal^2$-smooth \fcts\ $\psi_1,\ldots,\psi_k$ on $U$ such that $\psi=\max_{j=1,\ldots,k} \psi_j$ on $U$. By picking a slightly smaller \nbh\ of $p$, we can arrange that the \fcts\ $\psi_j$, $j=1,\ldots,k$, are defined in a \nbh\ of $\cl{U}$. Let $j_0$ be an index from $\{ 1,\ldots,k\}$ such that $\psi(p)=\psi_{j_0}(p)$. Since $\psi$ peaks at $p$, we have that
$$
\psi_{j_0}(p) = \psi(p) > \psi(z) \geq \psi_{j_0}(z)
$$
for every $z \in (U \cap K)\setminus\{ p\}$. Hence, $\psi_{j_0}$ peaks at $p$ in $K \cap U$. Since $\psi_{j_0}$ is \cont\ on $\cl{U}$, we can choose a suitable constant $c \in \rbb$ such that
$$
\psi_{j_0}(p) > c > \max_{bU \cap K} \psi_{j_0}.
$$
By Lemma \ref{lem-reg-max-qpsh} the \fct\ $\vphi:=\rmax{\eps}\{ \psi_{j_0},c\}$ is $\ccal^2$-smooth and \qpsh\ in a \nbh\ of $\cl{U} \cap K$. If we choose $\eps>0$ small enough, then due to Lemma \ref{lem-reg-max} (3) we can derive that the \fct\ $\vphi$ peaks at $p$ in $K$ and fulfills $\vphi=c$ on $bU \cap K$. In view of the previous property, we can extend $\vphi$ by the constant $c$ into a \nbh\ of $K$ in order to obtain a \fct\ from $\PSH_q^2(K)$ which peaks at $p$. Since $p$ was an arbitrary peak point for the class $\PSH_q^c(K)$, together with the inclusions in (\ref{eqqq}) above, we conclude that
$$
P_{\PSH_q^c(K)}= P_{\PSH_q^2(K)}.
$$
By the peak point property for the \qpsh\ \fcts\ from Definition \ref{def-more-qpsh} we obtain that
$$
\cs_{\PSH_q^2(K)} = \cs_{\PSH_q^c(K)}.
$$
Now Bungart's approximation theorem (see Corollary 5.4 in \cite{Bu}) and S\l odkowski's approximation theorem (see Theorem 2.9 in \cite{Sl2}) imply that
$$
\PSH_q^0(K) \subset \scl{\PSH_q^c(K)} \qand \PSH_q(K) \subset \scl{\PSH_q^0(K)}.
$$
Therefore, Corollary \ref{cor-shilov-cl} yields
$$
\cs_{\PSH_q(K)} \subset \cs_{\scl{\PSH_q^0(K)}}=\cs_{\PSH_q^0(K)} \subset \cs_{\scl{\PSH_q^c(K)}} = \cs_{\PSH_q^c(K)} \subset \cs_{\PSH_q(K)}.
$$
Hence, we obtain the remaining identities $\cs_{\PSH_q(K)}=\cs_{\PSH^0_q(K)}=\cs_{\PSH^c_q(K)}$.
\eproof

\brem\label{rem-local-global-peak} From the proof of the previous result we can derive the following local peak point property: Let $p$ be a boundary point of a compact set $K$ in $\cbb$. If $p$ is a local peak point for $\ccal^2$-smooth \qpsh\ \fcts, i.e., there is a \nbh\ $U$ of $p$ and a $\ccal^2$-smooth \qpsh\ \fct\ $\psi$ on $U$ such that $\psi(p)>\psi(z)$ for every $z \in (U\cap K)\setminus\{ p\}$, then $p$ is a (global) peak point for $\PSH_q^2(K)$.
\erem

We recall the definition of a strictly \qpsc\ boundary point of a smoothly bounded domain.

\begin{defn} Let $D$ be an open set in $\cbb^N$ with $\ccal^2$-smooth boundary, $p$ be a boundary point of $D$ and $q\in \{ 0,\ldots,N-1\}$. If there are a \nbh\ $U$ of $p$ and a $\ccal^2$-smooth strictly \qpsh\ \fct\ $\vrho$ on $U$ such that $d\vrho(p)\neq 0$ and $U \cap D = \{ z \in U : \vrho(z)<0\}$, then $D$ is said to be \textit{strictly \qpsc\ at $p$}. The set of all points $p \in bD$ such that $D$ is strictly \qpsc\ at $p$ is denoted by $\scal_q(\cld)$.
\end{defn}

Now we give a characterization of the Shilov boundary for \qpsh\ \fcts\ on bounded domains with $\ccal^2$-smooth boundary.

\begin{thm}\label{qshilov-qpsc} Let $q\in \{ 0,\ldots,N-1\}$ and let $D$ be a bounded domain in $\cbb^N$ with $\ccal^2$-smooth boundary. Then
$$
\cs_{\PSH_q(\cld)}=\cs_{\PSH^2_q(\cld)}= \cl{\scal_q(\cld)}.
$$
\end{thm}

\bproof It follows from Theorem 5.6 in \cite{HM} that
$$
P_{\PSH^2_q(D)\cap \ccal^0(\cld)} \subset \cl{\scal_q(\cld)} \qand \scal_q(\cld) \subset P_{\PSH_q(D)\cap \ccal^0(\cld)}.
$$
It follows from the definition that $\PSH_q^2(\cld) \subset \PSH^2_q(D)\cap \ccal^0(\cld)$ and, therefore,
\bea \label{qsqpsceq1}
P_{\PSH^2_q(\cld)}\subset \cl{\scal_q(\cld)}.
\eea
Hence, by the previous Proposition \ref{prop-more-qpsh-shilov} and the peak point property for $\ccal^2$-smooth \qpsh\ \fcts\ (see Remark \ref{rem-qshilov-2}) we obtain that
\bea \label{qsqpsceq2}
\cs_{\PSH^2_q(\cld)}=\cs_{\PSH_q(\cld)}=\cl{P}_{\PSH^2_q(\cld)} \subset \cl{\scal_q(\cld)}.
\eea
On the other hand, let $p \in \scal_q(\cld)$. Then there is a \nbh\ $U$ of $p$ and a $\ccal^2$-smooth strictly \qpsh\ \fct\ $\rho$ on $U$ such that $\rho$ vanishes on $bD \cap U$ and $\rho(z)<0$ if $z \in U \cap D$. Since $\rho$ is strictly \qpsh, there is a positive constant $\eps>0$ such that $\vphi(z):=\rho(z)-\eps|z-p|^2$ is $\ccal^2$-smooth and \qpsh\ on $U$. Moreover, $\vphi(p)=0$ and $\vphi(z)<0$ for every $z \in (U\cap \cld)\setminus\{ p\}$. In view of Remark \ref{rem-local-global-peak}, the point $p$ is also a peak point for the family $\PSH_q^2(\cld)$. Since $p$ is an arbitrary point from $\scal_q(\cld)$, it follows that $\scal_q(\cld)$ lies in $P_{\PSH^2_q(\cld)}$. In view of inclusion (\ref{qsqpsceq1}) above and the peak point property for $\ccal^2$-smooth \qpsh\ \fcts, we obtain that
$$
\cs_{\PSH_q^2(\cld)}=\cl{P}_{\PSH_q^2(\cld)}=\cl{\scal_q(\cld)}.
$$
The statement follows now from the inclusions (\ref{qsqpsceq2}) above.
\eproof

For the special case $q=N-1$ we obtain the following improvement of Remark \ref{rem-trivial-qshilov} and Corollary \ref{CN-1byckov} in the case of smoothly bounded domains.

\bthm \label{thm-N-1-qshilov} If $D$ is a bounded domain in $\cbb^N$ with $\ccal^2$-smooth boundary, then we have that
$$
P_{\PSH^2_{N-1}(\cld)} =\cs_{\PSH_{N-1}^2(\cld)} =bD.
$$
\ethm

\bproof Since $D$ is bounded and has a $\ccal^2$-smooth boundary, it is easy to construct a global defining \fct\ $\vrho$ for $D$, i.e., a $\ccal^2$-smooth \fcts\ in a \nbh\ $U$ of $\cld$ such that $D=\{ z \in U : \vrho(z)<0 \}$ and $d\vrho \neq 0$ on $bD$. Then for a large enough constant $c>0$, the \fct\ $\psi:=e^{c \vrho}-1$ is strictly $(N-1)$-\psh\ and $\ccal^2$-smooth in a \nbh\ $V \relc U$ of $bD$. By shrinking $V$, we can assume that $\psi$ is defined in a \nbh\ of $\cl{V}$ in $U$. For an appropriate choice of positive constants $\delta>0$ and $b>0$ we have that $\delta |z|^2-b < \psi(z)$ for every $z \in bD$ and that $\psi(z) < \delta |z|^2 -b \quad \mathrm{for \ every} \ z \in bV\cap D$. For a positive number $\eta>0$ we put $\tilde\psi(z):=\rmax{\eta}\{\psi(z),\delta|z|^2-b\}$. Then, by Lemma \ref{lem-reg-max} (3), we can choose $\eta>0$ so small that $\tilde\psi(z)=\psi(z)$ for every $z$ in some \nbh\ of $bD$ in $V$ and such that $\tilde\psi(z)=\delta|z|^2-b$ for every $z$ in some \nbh\ of $bV$ in $\cl{V}\cap D$. But then we can extend $\tilde\psi(z)$ by $\delta|z|^2-b$ into $D\setminus V$. We denote this extension again by $\tilde\psi$. Observe that $\tilde\psi$ is now strictly $(N-1)$-\psh\ in some \nbh\ of $\cl{D}$. Therefore, for every boundary point $p$ of $D$ there is a positive constant $\eps=\eps(p)$ such that $\tilde\psi(z)-\eps|z-p|^2$ is $(N-1)$-\psh\ and $\ccal^2$-smooth in some \nbh\ of $\cl{D}$. Moreover, it peaks at $p$. Hence, we derive that
$$
bD \subset P_{\PSH_{N-1}^2(\cld)} \subset \cs_{\PSH_{N-1}^2(\cld)} \subset bD.
$$
\eproof

We also mention here the following result obtained by Basener in \cite{Ba2} (see Theorem 5).

\begin{thm} Let $q\in \{ 0,\ldots,N-1\}$. Then $\check{S}_{A_q(\cld)}$ is contained in $\cl{\scal_q(\cld)}$.
\end{thm}

\brem The lack of appropriate gluing techniques for \qhol\ \fcts\ does not permit to obtain a converse results, i.e., it remains an open question whether the inclusion $\check{S}_{A_q(\cld)} \supset \scal_q(\cld)$ is also true.
\erem

As in the convex case (see Theorem \ref{thm-fol-cvx}) we can find a complex foliation in some parts of the Shilov boundary for \qpsh\ \fcts\ on smoothly bounded domains. For further results on complex foliations of real submanifolds we refer to \cite{Free}.

\begin{thm}\label{thm-fol-psc} Let $q$ be an integer from $\{ 1,\ldots,N-1\}$ and let $D$ be a bounded \psc\ domain in $\cbb^N$ with $\ccal^2$-smooth boundary. Then the open part
$$
\fcal_q(\cld):=\inti_{bD}\left(\check{S}_{\PSH_q(\cld)}\setminus\check{S}_{\PSH_{q-1}(\cld)}\right)
$$
of the Shilov boundary for $\PSH_q(\cld)$ in $bD$ locally admits a foliation by complex $q$-dimensional submanifolds, provided it is not empty.
\end{thm}

\begin{proof} By Theorem \ref{qshilov-qpsc} we have that
\bea\label{folcvxeq1}
\fcal_q(\cld) = \inti_{bD}\left(\cl{\scal_q(\cld)}\setminus\cl{\scal_{q-1}(\cld)}\right) = \scal_q(\cld)\setminus\cl{\scal_{q-1}(\cld)}.
\eea
Given a defining \fct\ $\vrho$ of $D$, by the definition of the set $\scal_q(\cld)$, by pseudoconvexity of $D$ and by the identities (\ref{folcvxeq1}) above, for each point $p \in \fcal_q(\cld)$ the complex Hessian $L$ of $\vrho$ at $p$ has exactly $N-q-1$ positive and $q$ zero eigenvalues on the holomorphic tangent space $H_p bD$ to $bD$ at $p$. Then, by Theorem 1.1 in \cite{Free}, the set $\fcal_q(\cld)$ locally admits a foliation by complex $q$-dimensional submanifolds.
\end{proof}

\addcontentsline{toc}{section}{References}
\bibliographystyle{alpha}
%\bibliography{bib/bibliography}

\vspace{1cm}

T.~Pawlaschyk, \textsc{Department of Mathematics and Informatics, University of Wuppertal, Gaussstr. 20, 42119 Wuppertal, Germany}\par\nopagebreak
  \textit{E-mail address:} \texttt{pawlaschyk@math.uni-wuppertal.de}

\end{document}